\documentclass[preprint,10pt,number,sort&compress,longtitle]{elsarticle}

\usepackage{amssymb}
\usepackage{graphicx}
\usepackage{latexsym}

\usepackage[bookmarksnumbered=true]{hyperref}
\hypersetup{
        colorlinks = true,
        linkcolor = blue,
        citecolor = blue,
        urlcolor = blue}
\usepackage{multirow}
\usepackage{tikz}
\usepackage{rotating}
\usepackage{bm} 
\usepackage{enumitem} 
\usepackage[hang]{subfigure} 
\usepackage{wrapfig} 
\usepackage{tabularx} 
\usepackage{dcolumn} 
\usepackage{booktabs} 
\usepackage{listings} 
\usepackage{psfrag}
\usepackage{amssymb}
\usepackage{amsmath} 
\usepackage{bm} 
\usepackage{amsfonts}
\usepackage{mathdots}
\RequirePackage{marvosym}

\newcommand*{\Geo}{\operatorname*{Geo}}

\newcommand*{\IP}{\mathbb{P}}
\newcommand*{\IE}{\mathbb{E}}

\newcommand*{\IR}{\mathbb{R}}

\newcommand*{\IN}{\mathbb{N}}

\newcommand{\ind}{{\bf 1}}

\newcommand*{\F}{\mathcal{F}}

\newtheorem{theorem}{Theorem}[section]
\newtheorem{definition}[theorem]{Definition}
\newtheorem{lemma}[theorem]{Lemma}
\newtheorem{corollary}[theorem]{Corollary}
\newtheorem{remark}[theorem]{Remark}
\newtheorem{proposition}[theorem]{Proposition}
\newtheorem{example}[theorem]{Example}
\newproof{proof}{Proof} 
\newproof{Proof}{Proof of Theorem } 

\newcommand{\red}[1]{\textcolor{red}{#1}}

\journal{Journal of Multivariate Analysis}

\begin{document}

\begin{frontmatter}

\title{Multivariate geometric distributions, (logarithmically) monotone sequences, and infinitely divisible laws\\ \red{\large{(with erratum by Natalia Shenkman)}}}

\author{Jan-Frederik Mai}\ead{jan-frederik.mai@xaia.com} \author{Matthias Scherer}\ead{scherer@tum.de} \author{Natalia Shenkman\corref{cor1}}\ead{shenkman@tum.de}

\address{Department of Mathematics, Technische Universit{\"a}t M{\"u}nchen, Parkring 11, D-85748 Garching bei M{\"u}nchen, Germany}
\cortext[cor1]{\red{Natalia Shenkman is solely responsible for the corrections and improvements marked in red. This includes a drastically simplified proof, the removal of two nonsensical proofs and the addition of some relevant references. The original manuscript was published in the Journal of Multivariate Analysis (2013).}}

\begin{abstract}
Two stochastic representations of multivariate geometric distributions are ana\-lyzed, both are obtained by lifting the lack-of-memory (LM) property of the univariate geometric law to the multivariate case. On the one hand, the narrow-sense multivariate geometric law can be considered a discrete equivalent of the well-studied Marshall-Olkin exponential law. On the other hand, the more general wide-sense geometric law is shown to be characterized by the LM property and can differ significantly from its continuous counterpart, e.g., by allowing for negative pairwise correlations. 

For both families of distributions, their $d$-dimensional exchangeable subclass is characterized analytically via $d$-log-monotone, resp.\ $d$-monotone, sequences of parameters. Using this reparameterization, the subfamilies of distributions with conditionally i.i.d.\ components in the sense of de Finetti's theorem are determined. For these, a third stochastic construction based on a non-decreasing random walk is presented. The narrow-sense family is embedded in this construction when the increments of the involved random walk are infinitely divi\-sible. The exchangeable narrow-sense law is furthermore shown to exhibit the multivariate right tail increasing (MRTI) dependence.
\end{abstract}

\begin{keyword}

multivariate geometric law \sep lack-of-memory \sep exchangeability \sep completely monotone sequence \sep de Finetti's theorem \sep infinitely divisible law 
\MSC[2000] 62H05 \sep 62H20

\end{keyword}

\end{frontmatter}

\section{Introduction}
\label{sec:intro}

The univariate geometric (resp.\ exponential) law is well known to be the only discrete (resp.\ continuous) distribution obeying the lack-of-memory (LM) property. Multivariate extensions have been studied in both cases. In the exponential case, an analytical derivation based on the multivariate continuous local LM property and probabilistic derivations based on exogenous exponential shock models are provided in \cite{marshall67}. All derivations are shown to lead to the same distribution. 
In the geometric case, multivariate extensions obtained as the waiting times for outcomes in a sequence of multinomial trials are introduced in \cite{arnold75}. In particular, \cite{arnold75} defines ``a general multivariate geometric distribution'' and presents a link to the Marshall-Olkin multivariate exponential distribution via multivariate geometric compounding of exponential random variables. 
Whereas \cite{arnold75} deals with probabilistic constructions of the multivariate geometric law, \cite{marshall95} and \cite{nair97} are concerned with analytical derivations based on the LM property. \cite{nair97} develop new classes of multivariate discrete survival functions adapted from monotonic behavior of a multivariate failure rate. In the process, they identify the families of multivariate geometric distributions with no-aging characteristic. Using the notation of this reference, survival functions belonging to both the $IFR_3$ and $DFR_3$ classes are precisely the survival functions satisfying the multivariate discrete local LM property. The results of \cite{marshall95} imply that, in contrast to the exponential case, the class of geometric distributions obtained using an exogenous fatal-shock model in the sense of \cite{marshall67}, where the exponential shocks are replaced by geometric ones, does not comprise all solutions to the functional equation corresponding to the multivariate discrete local LM property. Therefore, they subdivide the set of multivariate geometric distributions into $\mathcal{G^N}$ and $\mathcal{G^W}$, abbreviating ``narrow-sense'' $(\mathcal{N})$ and ``wide-sense'' $(\mathcal{W})$, respectively. The fact that the geometric law requires a finer treatment compared with the exponential law is emphasized in \cite{marshall95}: ``\textit{the structure of families of multivariate geometric distributions is more complicated and more interesting than is the structure of families of multivariate exponential distributions.}''
\par
In the recent literature, some attention has been drawn to the exchangeable and extendible subfamilies of multivariate distributions. In general, it is very difficult to verify analytically that a $d$-variate function $F$ is a probabi\-lity distribution (resp. survival probability) function. The major difficulty is to prove rectangular inequalities that guarantee positivity of the probability measure. The restriction to permutation invariant functions, corresponding to exchangeable probability distributions, can facilitate this task and lead to stunning mathematical links. For instance, if $F$ is parameterized by a univariate function $f$, i.e.\ $F=F(f)$, the rectangular inequalities for $F$ might translate into convenient analytic properties of $f$. These analytic properties might provide an illuminating loop back to probability theory and shed light on the probabilistic construction of $F$. In this spirit, \cite{kimberling74,malov01,morillas05,neil09} characterize Archimedean copulas via monotonicity properties of their generator function. The latter is shown to be a characterizing transform of a univariate probability law which appears in the construction of Archimedean copulas. Similar results are known for spherical and elliptical distributions, see \cite{fang90}, p.\ 26 ff., and \cite{kingman72,schoenberg38}. \cite{mai09,mai11,ressel11,ressel12} provide a link between Marshall-Olkin exponential distributions and completely monotone sequences. Relying on the results of \cite{gnedin08}, the references \cite{mai09,mai11} reveal astonishing interrelations with the theory of L\'{e}vy subordinators. A rigorous study of the exchangeable subsets of $\mathcal{G^N}$ and $\mathcal{G^W}$ discloses similar interesting interrelations with completely monotone sequences, random walks, and infinitely divisible laws. The present article contributes to the existing literature in the following way:
\begin{itemize}
\item \textbf{Stochastic models with limited memory:} It is shown that the family of general multivariate geometric distributions defined in \cite{arnold75} comprises precisely all solutions to the functional equation corresponding to the local discrete LM property and, thus, coincides with $\mathcal{G^W}$ (see Theorem \ref{Uniq}). Further, $\mathcal{G^N}$ forms a proper subset of $\mathcal{G^W}$. The survival functions for both families are computed in closed form (see Theorem \ref{thm_a}).
\item \textbf{Exchangeable subclasses}: The exchangeable subclasses are rewritten in terms of parameter sequences, truncated versions of which determine all lower-dimensional marginal distributions (see Remark \ref{repar}). The set of $d$-dimensional exchangeable wide-sense distributions ($\mathcal{G^{W,X}}$) is characterized by $d$-monotone sequences ($\mathcal{M}_{d}$), and the set of $d$-dimensional exchangeable narrow-sense distributions ($\mathcal{G^{N,X}}$) is characterized by the (newly introduced) stronger notion of $d$-log-monotone sequences ($\mathcal{LM}_{d}$) (see Theorem \ref{exch}). 
\item \textbf{Subclasses with conditionally i.i.d.\ components}: Starting from a $d$-variate exchangeable distribution, we investigate whether the distribution in question can be obtained by truncation of an infinite exchangeable sequence. This is solved analytically and corresponds to extendibility of the parameter sequence to a completely monotone sequence ($\mathcal{M}_{\infty}$), respectively to a completely log-monotone sequence ($\mathcal{LM}_{\infty}$) (see Corollary \ref{ext}). For both the extendible wide-sense ($\mathcal{G^{W,E}}$) and the extendible narrow-sense ($\mathcal{G^{N,E}}$) families, a stochastic model with components that are i.i.d.\ conditioned on a non-decreasing random walk is constructed (see Theorem \ref{thm_ciid} \red{(Esary and Marshall (1973))}). In the universe of this model, the $\mathcal{G^{N,E}}$-distribution arises as the proper subclass constructed by random walks with infinitely divisible increments. A natural link between $\mathcal{LM}_{\infty}$ and infinitely divisible distributions is established (see Proposition~\ref{extend} \red{(Gnedin and Pitman (2008))}). Interconnections with the results of \cite{gnedin04,gnedin05} on random integer partitions are highlighted.   
\item \textbf{Dependence properties}: Pairwise correlations are computed for both the wide- and narrow-sense geometric distribution (see Lemma \ref{l:corr}). Interestingly, the wide-sense geometric distribution allows for negative correlations, whereas the narrow-sense geometric distribution can only model non-negative correlations. The exchangeable narrow-sense as well as the extendible wide-sense distributions are shown to exhibit the MRTI property (see Theorem \ref{mrti_exch} and Corollary \ref{mrti_ext}).
\item \textbf{Relations between subfamilies}: The following Venn diagram illustrates the relations that we identify between subfamilies of $\mathcal{G^{W,X}}$. 
\begin{figure}[!ht]
\begin{center}
\caption{Characterization of multivariate exchangeable geometric distributions.}
\begin{tikzpicture}
\path[fill=white,draw=black](0,0) rectangle (12,5);
\path[fill=white,draw=black](0.2,0.2) rectangle (11.8,4.3);
\draw(0.26,4.55) node[text=black,fill=white,anchor=base west]{\small{\text{$\mathcal{G^{W,X}}\;\hat{=}\;\mathcal{M}_{d}$} \emph{(allows for negative pairwise correlations)}}};
\draw(0.26,3.45) node[text=black,fill=white,anchor=base west]{\small{\text{$\mathcal{G^{W,X}}$ \emph{MRTI positive dependence}}}}; 
\draw(0.55,1.349)  node[text=black,fill=white,anchor=base west]{\small{\emph{conditioned on a random walk}}};
\draw(1.29,1.8)  node[text=black,fill=white,anchor=base west]{\small{\emph{i.i.d. representation}}};
\draw(7.3,3.45) node[text=black,fill=white,anchor=base west]{\small{\text{$\mathcal{G^{N,X}}\;\hat{=}\;\mathcal{LM}_{d}$}}};
\draw(1.7,2.251) node[text=black,fill=white,anchor=base west]{\small{\text{$\mathcal{G^{W,E}}\;\hat{=}\;\mathcal{M}_{\infty}$}}};
\draw(5.9,1.349)  node[text=black,fill=white,anchor=base west]{\small{\emph{conditioned on a random walk}}};
\draw(6.7,1.8)  node[text=black,fill=white,anchor=base west]{\small{\emph{i.i.d. representation}}};
\draw(7.3,2.251) node[text=black,fill=white,anchor=base west]{\small{\text{$\mathcal{G^{N,E}}\;\hat{=}\;\mathcal{LM}_{\infty}$}}};
\draw(5.6,0.898)  node[text=black,fill=white,anchor=base west]{\small{\emph{with infinitely divisible increments}}};
\path[draw=black](5.6,0.6) rectangle (11.5,4);
\path[draw=black](0.43,0.4) rectangle (11.2,3);
\label{tab_intro}
\end{tikzpicture}
\end{center}
\end{figure}
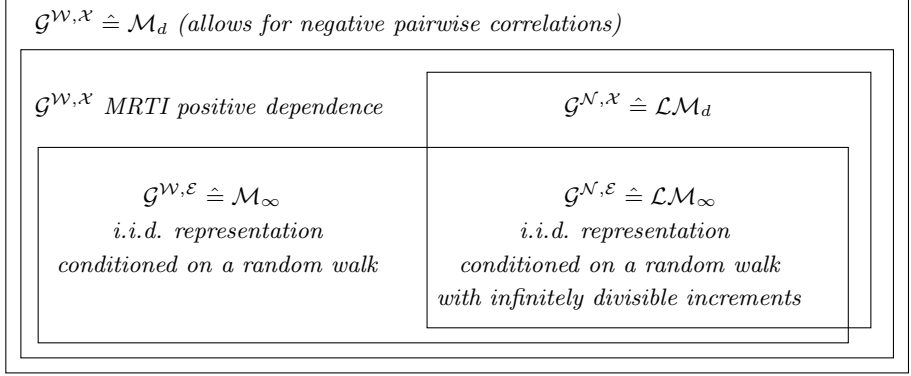
In dimension $2$, and only in dimension $2$, the inner sets agree, i.e.\ 
\begin{gather*}
\hspace{-0.2cm}\mathcal{G^{N,E}}\hspace{-0.15cm}=\mathcal{G^{N,X}}\hspace{-0.15cm}=\mathcal{G^{W,E}}\hspace{-0.15cm}=\mathcal{G^{W,X}_{\;\text{MRTI}}}=\mathcal{G^{W,X}}\text{with non-negative correlations},
\end{gather*}
and form a proper subset of $\mathcal{G^{W,X}}$. 
\end{itemize}

\section{Multivariate geometric distributions with limited memory}
\label{sec:mgd}

This section derives multivariate extensions of the univariate geometric distribution satisfying the local discrete LM property. To this end, we consider random variables $\tau$ taking values in $\IN:=\{1,2,\ldots\}$ and denote $\IN_0:=\{0\} \cup \IN$. The following definition is included for the sake of clarity.

\begin{definition}[Discrete survival function] \label{def_surv} For $d\in\IN$, the multi-indexed sequence $\{\bar{F}_{n_1,\ldots,n_d}\}_{n_1,\ldots,n_d \in \IN_0}$ is called a \emph{discrete survival function} if there exists a probability space $(\Omega,\F,\IP)$ supporting an $\IN^d$-valued random vector $(\tau_1,\ldots,\tau_d)$ such that
\begin{gather*}
\bar{F}_{n_1,\ldots,n_d}=\IP(\tau_1>n_1,\ldots,\tau_d>n_d),\quad n_1,\ldots,n_d \in \IN_0.
\end{gather*}

\end{definition}

A prominent example of a distribution on $\IN$ and the univariate starting point of the present investigation is the geometric law.

\begin{example}[Geometric distribution] \label{ex:geom}
Let $p\in[0,1)$. Considering a series of i.i.d.\ Bernoulli distributed random variables $Z_1,Z_2,\ldots$ with success probability $1-p$, i.e.\ $\IP(Z_1=1)=1-p=1-\IP(Z_1=0)$, we denote by $\tau:=\min\{n \in \IN\,:\,Z_n=1\}$ the first time of success. The distribution of $\tau$ is called \emph{geometric} and denoted by $Geo(1-p)$. Its corresponding survival function $\{\bar{F}_n\}_{n \in \IN_0}$ is given by $\bar{F}_n:=\IP(\tau>n)=p^{n}$, $n \in \IN_0$.
\end{example}
As already mentioned, the geometric distribution is the only distribution on $\IN$ satisfying the discrete LM property $\bar{F}_{n+m}=\bar{F}_n\,\bar{F}_m$, $n,m\in \IN_0.$ This renders the geometric distribution a discrete equivalent of the (continuous) exponential distribution. Moreover, it is easy to see that the geometric distribution is min-stable in the following sense: if $\tau_1 \sim Geo(1-p_1)$ and $\tau_2 \sim Geo(1-p_2)$ are independent, then $\min\{\tau_1,\tau_2\} \sim Geo(1-p_1\,p_2)$. This is another property shared with the exponential law.
Considering limiting cases, $\tau\sim Geo(1)$ implies $\tau=1$ almost surely. Further, a distribution with a point mass at infinity is conveniently interpreted as $Geo(0)$. 
\par
A first multivariate extension of the geometric law, preserving the local discrete LM property, is obtained using a discrete analogue of the Marshall--Olkin exponential shock model, see \cite{marshall67}, where the exponentially distributed shocks are replaced by geometric ones. The resulting distribution is referred to as a multivariate geometric distribution in the narrow sense ($\mathcal{G^N}$). In the continuous case, the Marshall--Olkin exponential shock model generates precisely the class of all multivariate exponential distributions satisfying the local continuous LM property. However, the family of multivariate geometric distributions satisfying the local LM property has a more complicated and richer structure. In particular, it contains distributions which can not be obtained using a discrete analogue of the Marshall--Olkin model. 
\par
As a second multivariate extension of the geometric law, we study a multivariate geometric distribution introduced in Section 4 of \cite{arnold75}. We prove that this class of distributions comprises all solutions to the functional equation corresponding to the local discrete multivariate LM property. Consequently, it forms a proper superset of the parametric family of narrow-sense geometric distributions. Adopting the terminology of \cite{marshall95}, geometric distributions obtained using the Arnold model will be called geometric in the wide sense ($\mathcal{G^W}$). 

\begin{definition}[Narrow- and wide-sense multivariate geometric law]\label{def}\text{}
\vspace{-0.7cm}
\begin{itemize}
\item[($\mathcal{N}$)] Let $p_I \in [0,1]$, $\emptyset \neq I \subseteq \{1,\ldots,d\}$, with $\prod_{I:k \in I}p_I<1$ for $k=1,\ldots,d$, be given parameters. Let $(\Omega,\F,\IP)$ be a probability space supporting a collection $\{E_I\}_{\emptyset \neq I \subseteq \{1,\ldots,d\}}$ of independent geometric random variables, where $E_I \sim Geo(1-p_I)$. Define the random vector $(\tau_1,\ldots,\tau_d)$ by
\begin{gather*}
\tau_k:=\min\{E_I\,:\,k\in I\},\quad k=1,\ldots,d.
\end{gather*}
The distribution of $(\tau_1,\ldots,\tau_d)$ is called \emph{multivariate narrow-sense geometric} with parameters $\textbf{p}:=\{p_I\}_{\emptyset \neq I \subseteq \{1,\ldots,d\}}$ and is denoted by $\mathcal{G^N(\textbf{p})}$. 
\item[($\mathcal{W}$)] Let $\tilde{p}_I \in [0,1]$, $I\subseteq \{1,\ldots,d\}$, with $\sum_{I}\tilde{p}_{I}=1$ and $\sum_{I:k \notin I}\tilde{p}_{I}<1$ for $k=1,\ldots,d$, be given parameters. Consider an experiment with $2^d$ possible outcomes $I\subseteq \{1,\ldots,d\}$, each of which occurs with probability $\tilde{p}_I$. The experiment is run repeatedly in independent and identically distributed trials. Further, let $\{\widetilde{E}_I\}_{I \subseteq \{1,\ldots,d\}}$ be a collection of random variables, where $\widetilde{E}_I$ represents the number of the experiment which first yielded the outcome $I$. Define the random vector $(\tau_1,\ldots,\tau_d)$ by
\begin{gather*}
\tau_k:=\min\{\widetilde{E}_I\,:\,k\in I\},\quad k=1,\ldots,d.
\end{gather*}
The distribution of $(\tau_1,\ldots,\tau_d)$ is called \emph{multivariate wide-sense geometric} with parameters $\tilde{\textbf{p}}:=\{\tilde{p}_I\}_{I \subseteq \{1,\ldots,d\}}$ and is denoted by $\mathcal{G^W(\tilde{\textbf{p}})}$.     
\end{itemize}
\end{definition}

\begin{remark} [Comparing the stochastic models]\text{}
\vspace{-0.2cm}
\begin{enumerate}
\item [($\mathcal{N}$)] In the case of a narrow-sense geometric distribution, each component $\tau_k$, $k=1,\ldots,d$, is defined as a minimum over $2^{d-1}$ independent geometric random variables indexed by the subsets $\emptyset \neq I \subseteq \{1,\ldots,d\}$ with $k\in I$. By min-stability of the geometric law, $\tau_k$ is $Geo(1-\prod_{I:k \in I}p_I)$-distributed. The condition $\prod_{I:k \in I}p_I<1$ ensures that $\tau_k$ is almost surely finite. The case where at least one $\tau_k$ is almost surely $1$ can be excluded by requiring $p_I\neq0$ for all $I$. 
The dependence between the components results from the overlapping subset indices in the definition of the $\tau_k$'s. 
\item [($\mathcal{W}$)] In the case of a wide-sense geometric distribution, the shocks are geo\-metric $\widetilde{E}_I\sim Geo(\tilde{p}_I)$, $I \subseteq \{1,\ldots,d\}$, and $\tau_k\sim Geo(1-\sum_{I:\;k\notin I}\tilde{p}_I)$. Again, each $\tau_k$ is almost surely finite due to $\sum_{I:k \notin I}\tilde{p}_{I}<1$ for $k=1,\ldots,d$, and $\tau_k=1$ almost surely if $\sum_{I:k \notin I}\tilde{p}_{I}=0$. The crucial difference between the two probabilistic models is that the geometric shocks $E_I$ in the definition of the narrow-sense geometric distribution are \emph{independent}, whereas the construction of a wide-sense multivariate geometric distribution is based on \emph{dependent} geometric random variables $\widetilde{E}_I$. Due to $\sum_{I}\tilde{p}_{I}=1$ in the case of a wide-sense geometric distribution, both distributions $\mathcal{G^N}$ and $\mathcal{G^W}$ are parameterized by $2^d-1$ parameters. 
\end{enumerate}
\end{remark}

Theorems \ref{thm_a} and \ref{Uniq} characterize both distributions in terms of the survival function and the local LM property. In the following, $n_{(1)}\leq n_{(2)}\leq\ldots\leq n_{(d)}$ denotes the ordered list of $n_{1},\ldots,n_{d}\in\IN_0$. 

\begin{theorem}[Survival function and limited memory] \label{thm_a}\text{}
\vspace{-0.2cm}
\begin{itemize}
\item[($\mathcal{N}$)] Let $(\tau_1,\ldots,\tau_d)\sim\mathcal{G^N(\textbf{p})}$. Then, $(\tau_1,\ldots,\tau_d)$ has support $\IN^d$ with discrete survival function $\{\bar{F}^{\mathcal{N}}_{n_1,\ldots,n_d}\}_{n_1,\ldots,n_d \in \IN_0}$, given by
\begin{gather}
\bar{F}^{\mathcal{N}}_{n_1,\ldots,n_d}:=\IP(\tau_1>n_1,\ldots,\tau_d>n_d)=\prod_{\emptyset \neq I \subseteq \{1,\ldots,d\}} p_{I}^{\max_{i \in I}\{n_i\}}.
\label{Geo_surv}
\end{gather}
\item[($\mathcal{W}$)] Let $(\tau_1,\ldots,\tau_d)\sim\mathcal{G^W(\tilde{\textbf{p}})}$. Then, $(\tau_1,\ldots,\tau_d)$ has support $\IN^d$ with discrete survival function $\{\bar{F}^{\mathcal{W}}_{n_1,\ldots,n_d}\}_{n_1,\ldots,n_d \in \IN_0}$, given by
\begin{gather}
\bar{F}^{\mathcal{W}}_{n_1,\ldots,n_d}\hspace{-0.05cm}:=\IP(\tau_1\hspace{-0.05cm}>\hspace{-0.05cm}n_1,\ldots,\tau_d\hspace{-0.05cm}>\hspace{-0.05cm}n_d)=\prod_{k=1}^d\Big(\hspace{-0.25cm}\sum_{\stackrel{I\subseteq \{1,\ldots,d\}}{\pi_{\textbf{n}}(i)\notin I\;\forall\,i=k,\ldots,d}}\hspace{-0.25cm} \tilde{p}_{I}\Big)^{n_{(k)}-n_{(k-1)}},
\label{Geo_surv_w}
\end{gather}
with the convention $n_{(0)}:=0$, and $\pi_{\textbf{\textit{n}}}$: $\{1,\ldots,d\}\rightarrow\{1,\ldots,d\}$ being a permutation depending on $\textbf{\textit{n}}=(n_1,\ldots,n_d)$ such that $n_{\pi_\textbf{n}(1)}\leq n_{\pi_\textbf{n}(2)}\leq\ldots\leq n_{\pi_\textbf{n}(d)}$.
\end{itemize}
Moreover, both distributions (\ref{Geo_surv}) and (\ref{Geo_surv_w}) satisfy the \emph{local discrete multivariate LM property} 
\hspace{-0.2cm}$$\IP(\tau_{i_1}\hspace{-0.1cm}>\hspace{-0.1cm}n_{i_1}+m,\ldots,\tau_{i_k}\hspace{-0.1cm}>\hspace{-0.1cm}n_{i_k}+m\,|\,\tau_{i_1}\hspace{-0.1cm}>\hspace{-0.1cm}m,\ldots,\tau_{i_k}\hspace{-0.1cm}>\hspace{-0.1cm}m)=\IP(\tau_{i_1}\hspace{-0.1cm}>\hspace{-0.1cm}n_{i_1},\ldots,\tau_{i_k}\hspace{-0.1cm}>\hspace{-0.1cm}n_{i_k}),$$ 
\vspace{-1cm}
\begin{gather}
\mbox{ for }k =1,\ldots,d,\,1\leq i_1,\ldots,i_k \leq d,\,m,n_{i_1},\ldots,n_{i_k} \in \IN_0.
\label{LM}
\end{gather}
\end{theorem}
\begin{proof} Refer to \ref{sec:appa}.\qed
\end{proof}
  
\begin{theorem}[Uniqueness] \label{Uniq}
A distribution on $\IN^d$ satisfies the local discrete LM property (\ref{LM}) if and only if it is a $d$-variate wide-sense geometric distribution.
\end{theorem}
\begin{proof} Refer to \ref{sec:appa}.\qed
\end{proof}

The inclusion $\mathcal{G^N}\subsetneq\mathcal{G^W}$ is immediate from Theorems \ref{thm_a} and \ref{Uniq}, and proper in any dimension $d\geq2$, see Example \ref{nw1}. Regarding a probabilistic meaning of the parameters $\tilde{\textbf{p}}=\{\tilde{p}_I\}_{I \subseteq \{1,\ldots,d\}}$ of $(\tau_1,\ldots,\tau_d)\sim\mathcal{G^W(\tilde{\textbf{p}})}$, we have
\begin{gather*}
\tilde{p}_I=\IP(\widetilde{E}_I=1)=\IP\big(\{\tau_i>1\;\forall i\notin I\}\cap\{\tau_i=1\;\forall i\in I\}\big),\quad I\subseteq \{1,\ldots,d\}.
\end{gather*}

\begin{example}[Three-variate geometric law]\label{mgd}\text{}
\vspace{-0.2cm}
\begin{itemize}
\item[($\mathcal{N}$)] Let $(\tau_1,\tau_2,\tau_3)\sim\mathcal{G^N}(\{p_I\}_{\emptyset \neq I \subseteq \{1,2,3\}})$. Then, its corresponding survival function $\{\bar{F}^{\mathcal{N}}_{n_1,n_2,n_3}\}_{n_1,n_2,n_3 \in \IN_0}$ is given by
\begin{gather}
\bar{F}^{\mathcal{N}}_{n_1,n_2,n_3}=p_{\{1\}}^{n_1}p_{\{2\}}^{n_2}p_{\{3\}}^{n_3}p_{\{1,2\}}^{\max\{n_1,n_2\}}p_{\{1,3\}}^{\max\{n_1,n_3\}}p_{\{2,3\}}^{\max\{n_2,n_3\}}p_{\{1,2,3\}}^{n_{(3)}}.
\label{Geo_three_n}
\end{gather}
\item[($\mathcal{W}$)] Let $(\tau_1,\tau_2,\tau_3)\sim\mathcal{G^W}(\{\tilde{p}_I\}_{I \subseteq \{1,2,3\}})$. Then, its corresponding survival function $\{\bar{F}^{\mathcal{W}}_{n_1,n_2,n_3}\}_{n_1,n_2,n_3 \in \IN_0}$ is given by 
\begin{gather}
\hspace{-0.85cm}\bar{F}^{\mathcal{W}}_{n_1,n_2,n_3}\hspace{-0.05cm}\hspace{-0.15cm}=\hspace{-0.1cm}\left\{\begin{array}{llllll}
     \hspace{-0.15cm}\tilde{p}_{\emptyset}^{n_1}(\tilde{p}_{\emptyset} \hspace{-0.1cm}+  \hspace{-0.1cm}\tilde{p}_{\{1\}} \hspace{-0.05cm})^{n_2-n_1}(\tilde{p}_{\emptyset}\hspace{-0.1cm}+\hspace{-0.1cm}\tilde{p}_{\{1\}}\hspace{-0.1cm}+\hspace{-0.1cm}\tilde{p}_{\{2\}}\hspace{-0.1cm}+\hspace{-0.1cm}\tilde{p}_{\{1,2\}}\hspace{-0.05cm})^{n_3-n_2},    &\mbox\; \hspace{-0.1cm}n_1\hspace{-0.1cm}\leq \hspace{-0.1cm}n_2 \hspace{-0.1cm}\leq \hspace{-0.1cm}n_3,\\ \hspace{-0.15cm}\tilde{p}_{\emptyset}^{n_1}(\tilde{p}_{\emptyset}\hspace{-0.1cm}+\hspace{-0.1cm}\tilde{p}_{\{1\}}\hspace{-0.05cm})^{n_3-n_1}(\tilde{p}_{\emptyset}\hspace{-0.1cm}+\hspace{-0.1cm}\tilde{p}_{\{1\}}\hspace{-0.1cm}+\hspace{-0.1cm}\tilde{p}_{\{3\}}\hspace{-0.1cm}+\hspace{-0.1cm}\tilde{p}_{\{1,3\}}\hspace{-0.05cm})^{n_2-n_3},    &\mbox\; \hspace{-0.1cm}n_1\hspace{-0.1cm}\leq \hspace{-0.1cm}n_3 \hspace{-0.1cm}< \hspace{-0.1cm}n_2,\\
     \hspace{-0.15cm}\tilde{p}_{\emptyset}^{n_2}(\tilde{p}_{\emptyset}\hspace{-0.1cm}+ \hspace{-0.1cm}\tilde{p}_{\{2\}}\hspace{-0.05cm})^{n_1-n_2}(\tilde{p}_{\emptyset}\hspace{-0.1cm}+\hspace{-0.1cm}\tilde{p}_{\{1\}}\hspace{-0.1cm}+\hspace{-0.1cm}\tilde{p}_{\{2\}}\hspace{-0.1cm}+\hspace{-0.1cm}\tilde{p}_{\{1,2\}}\hspace{-0.05cm})^{n_3-n_1},    &\mbox\; \hspace{-0.1cm}n_2\hspace{-0.1cm}< \hspace{-0.1cm}n_1 \hspace{-0.1cm}\leq \hspace{-0.1cm}n_3,\\
     \hspace{-0.15cm}\tilde{p}_{\emptyset}^{n_2}(\tilde{p}_{\emptyset}\hspace{-0.1cm}+\hspace{-0.1cm} \tilde{p}_{\{2\}}\hspace{-0.05cm})^{n_3-n_2}(\tilde{p}_{\emptyset}\hspace{-0.1cm}+\hspace{-0.1cm}\tilde{p}_{\{2\}}\hspace{-0.1cm}+\hspace{-0.1cm}\tilde{p}_{\{3\}}\hspace{-0.1cm}+\hspace{-0.1cm}\tilde{p}_{\{2,3\}}\hspace{-0.05cm})^{n_1-n_3},    &\mbox\; \hspace{-0.1cm}n_2\hspace{-0.1cm}\leq \hspace{-0.1cm}n_3 \hspace{-0.1cm}< \hspace{-0.1cm}n_1,\\
     \hspace{-0.15cm}\tilde{p}_{\emptyset}^{n_3}(\tilde{p}_{\emptyset}\hspace{-0.1cm}+ \hspace{-0.1cm}\tilde{p}_{\{3\}}\hspace{-0.05cm})^{n_1-n_3}(\tilde{p}_{\emptyset}\hspace{-0.1cm}+\hspace{-0.1cm}\tilde{p}_{\{1\}}\hspace{-0.1cm}+\hspace{-0.1cm}\tilde{p}_{\{3\}}\hspace{-0.1cm}+\hspace{-0.1cm}\tilde{p}_{\{1,3\}}\hspace{-0.05cm})^{n_2-n_1},    &\mbox\; \hspace{-0.1cm}n_3\hspace{-0.1cm}<\hspace{-0.1cm} n_1 \hspace{-0.1cm}\leq \hspace{-0.1cm}n_2,\\
     \hspace{-0.15cm}\tilde{p}_{\emptyset}^{n_3}(\tilde{p}_{\emptyset}\hspace{-0.1cm}+\hspace{-0.1cm} \tilde{p}_{\{3\}}\hspace{-0.05cm})^{n_2-n_3}(\tilde{p}_{\emptyset}\hspace{-0.1cm}+\hspace{-0.1cm}\tilde{p}_{\{2\}}\hspace{-0.1cm}+\hspace{-0.1cm}\tilde{p}_{\{3\}}\hspace{-0.1cm}+\hspace{-0.1cm}\tilde{p}_{\{2,3\}}\hspace{-0.05cm})^{n_1-n_2},    &\mbox\; \hspace{-0.1cm}n_3\hspace{-0.1cm}< \hspace{-0.1cm}n_2\hspace{-0.1cm} < \hspace{-0.1cm}n_1.\\
\end{array}\right.
\label{Geo_three_w}
\end{gather}
\end{itemize}
An elementary example of a three-variate narrow-sense geometric distribution is $(\tau_1,\tau_2,\tau_3)$ with mutually independent $\tau_i\sim Geo(1-p_i)$, $p_i\in[0,1)$, $i=1,2,3$, i.e. 
\begin{gather}
\bar{F}_{n_1,n_2,n_3}=p_1^{n_1}p_2^{n_2}p_3^{n_3}.
\label{Geo_ind}
\end{gather}
The survival function in (\ref{Geo_ind}) can be represented as in (\ref{Geo_three_n}) with parameters of a narrow-sense distribution $p_{\{i\}}=p_i$, $i=1,2,3$, and $p_I=0$ for $|I|>1$ . (\ref{Geo_ind}) can also be rewritten to resemble the survival function of a wide-sense distribution (\ref{Geo_three_w}) by setting $\tilde{p}_I=\sum_{i=0}^{|I|}(-1)^{|I|+i}\sum_{\stackrel{J\subseteq I}{|J|=i}}\prod_{\stackrel{j\in\{1,2,3\}}{j\notin J}}p_j$, $I \subseteq \{1,2,3\}$. 
\end{example}

For a discussion of the bivariate wide-sense geometric distribution, we refer the reader to \cite{azlarov83}.
      
Studying the exchangeable subfamilies of $\mathcal{G^N}$ and $\mathcal{G^W}$ reveals astonishing interrelations with completely monotone sequences and infinitely divisible laws. In particular, a convenient analytical characterization can be achieved for exchangeable geometric survival functions. Furthermore, exchangeable distributions might be of interest for practical reasons: in the non-exchangeable case, the number of parameters, namely $2^d-1$ in dimension $d$, grows exponentially in $d$. The restriction to permutation invariant distributions allows to significantly reduce the number of parameters to $d$ in dimension $d$.

\section{Exchangeable subfamilies}
\label{sec:exch}

This section provides an analytical characterization of exchangeable narrow- and wide-sense geometric distributions, i.e.\ distributions invariant under all permutations of their components. An exchangeability criterion for the distributions in (\ref{Geo_surv}) and (\ref{Geo_surv_w}) is given in Lemma \ref{lemma_exch_case}.  

\begin{lemma}[Exchangeability criterion] \label{lemma_exch_case}
The $d$-dimensional survival functions (\ref{Geo_surv}) and (\ref{Geo_surv_w}) are exchangeable if and only if their parameters satisfy the following condition:
\begin{gather}
|I_1|=|I_2|\; \Rightarrow \; p_{I_1}=p_{I_2},\text{ resp. } |I_1|=|I_2|\; \Rightarrow \; \tilde{p}_{I_1}=\tilde{p}_{I_2},
\label{exch_condi} 
\end{gather}
where $|I|$ denotes the cardinality of a set $I \subseteq \{1,\ldots,d\}$.
\end{lemma}
\begin{proof} Refer to \ref{sec:appb}. \red{While the lemma is intuitive,  the proofs provided in the appendix convey the impression of a certain complexity. They become trivial if the correct parametrization of the survival functions is invoked; see  \cite{shenkman17} and the proof of Lemma $2.1$ in \cite{shenkman20}.} \qed 
\end{proof}

Lemma \ref{lemma_exch_case} implies that in the exchangeable case, both distributions (\ref{Geo_surv}) and (\ref{Geo_surv_w}) are parameterized by only $d$ parameters. Remark \ref{repar} carries out a reparametrization of exchangeable survival functions in terms of parameter sequences, truncated versions of which determine all lower-dimensional marginal distributions. Besides notational convenience, this reparametrization facilitates an analytical characterization of the corresponding survival function.  

\begin{remark}[Reparameterization] \label{repar}\text{}
\vspace{-0.2cm}
\begin{enumerate}
\item [($\mathcal{N}$)] For the exchangeable narrow-sense geometric distribution, the distinct parameters are $p_1,\ldots,p_d$, where $p_k:=p_{\{1,\ldots,k\}}$ for $k=1,\ldots,d$. These parameters satisfy $p_1,\ldots,p_d\in (0,1]$ and $\prod_{i=1}^{d}p_i<1$. A combinatorial exercise shows that in this case, the survival function (\ref{Geo_surv}) simplifies to
\begin{gather*}
\bar{F}^{\mathcal{N}}_{n_1,\ldots,n_d}=\prod_{k=1}^{d}\Big( \prod_{i=1}^{d-k+1}p_i^{\binom{d-k}{i-1}} \Big)^{n_{(d-k+1)}}.
\end{gather*}
Introducing the notation
\begin{gather}
a_k:=\prod_{i=1}^{d-k+1}p_i^{\binom{d-k}{i-1}},\quad k=1,\ldots,d,
\label{ak}
\end{gather}
we obtain a convenient expression for the $d$-dimensional exchangeable narrow-sense geometric survival function:
\begin{gather}
\bar{F}^{\mathcal{N}}_{n_1,\ldots,n_d}=\prod_{k=1}^{d}a_k^{n_{(d-k+1)}}.
\label{Geo_surv_new}
\end{gather}
For a comparison of the narrow- with the wide-sense geometric distributions, it is more convenient to parameterize the survival function (\ref{Geo_surv_new}) in terms of the parameters $\{b_k\}_{k=1}^d$, given by    
\begin{gather}
b_{k}:=\prod_{i=1}^{k}a_i,\quad k=1,\ldots,d.
\label{b's}
\end{gather}
Denoting $b_0:=1$ and $n_{(0)}:=0$, this gives $a_k=b_{k}/b_{k-1}$ for $k=1,\ldots,d$, and the survival function in (\ref{Geo_surv_new}) takes the form 
\begin{gather}
\bar{F}^{\mathcal{N}}_{n_1,\ldots,n_d}=\prod_{k=1}^{d}(b_k/b_{k-1})^{n_{(d-k+1)}}=\prod_{k=1}^{d}b_k^{n_{(d-k+1)}-n_{(d-k)}}.
\label{kurzi}
\end{gather}
\item [($\mathcal{W}$)] Analogously, the exchangeable wide-sense geometric distribution can be parame\-terized by $d$ parameters $\tilde{p}_1,\ldots,\tilde{p}_d$, where $\tilde{p}_1:=\tilde{p}_{\emptyset}$ and $\tilde{p}_k:=\tilde{p}_{\{1,\ldots,k-1\}}$ for $k=2,\ldots,d$. The parameter $\tilde{p}_{\{1,\ldots,d\}}$ is determined by $\tilde{p}_1,\ldots,\tilde{p}_d$ via the condition $\sum_I \tilde{p}_I=1$. The parameters $\tilde{p}_1,\ldots,\tilde{p}_d$ satisfy $\tilde{p}_1,\ldots,\tilde{p}_d\in [0,1]$, $\sum_{i=1}^{d}\binom{d}{i-1}\tilde{p}_i\leq1$, and $\sum_{i=1}^{d}\binom{d-1}{i-1}\tilde{p}_i<1$. Further, the exchangeable wide-sense geometric survival function (\ref{Geo_surv_w}) can be conveniently written as
\begin{gather}
\bar{F}^{\mathcal{W}}_{n_1,\ldots,n_d}=\prod_{k=1}^{d}\beta_k^{n_{(d-k+1)}-n_{(d-k)}},
\label{Geo_surv_new_w}
\end{gather}
where the new parameters $\{\beta_k\}_{k=1}^d$ are given by    
\begin{gather}
\beta_{k}:=\sum_{i=1}^{d-k+1}\binom{d-k}{i-1}\tilde{p}_i,\quad k=1,\ldots,d.
\label{beta's}
\end{gather}
\end{enumerate}
\end{remark}

There exists a fundamental connection between the parameter spaces of exchangeable multivariate narrow- and wide-sense geometric distributions and $d$-monotone sequences. On the one hand, the sequence of parameters $(1,\beta_1,\ldots,\beta_d)$, given as in (\ref{beta's}), of every exchangeable $d$-variate wide-sense geometric distribution is $(d+1)$-monotone. On the other hand, each $(d+1)$-monotone sequence defines admissible parameters of an exchangeable $d$-variate geometric distribution satisfying the local LM property. The related concept of complete monotonicity has been studied in the literature, see e.g.\ \cite{lorch83,gnedin08}. In the present article, we particularize $d$-monotonicity by introducing the stronger notion of $d$\textit{-log-monotonicity}. The meaning of $d$-log-monotonicity is twofold. First, it characterizes the parameters $\{b_k\}_{k=1}^d$, given as in (\ref{b's}), of an exchangeable $d$-variate narrow-sense geometric distribution. Second, as we shall see in Section \ref{sec:ext}, the complete log-monotonicity, i.e.\ $d$-log-monotonicity for all $d\in\IN_0$, uniquely characterizes the sequence of non-positive exponential moments of an \textit{infinitely divisible} probability law on $[0,\infty]$. This result constitutes an analogue to Hausdorff's moment problem, see \cite{hausdorff21,hausdorff23}, i.e.\ $\{\beta_k\}^{\infty}_{k=0}$ is completely monotone if and only if it is the sequence of moments of a probability measure on $[0,1]$. Moreover, it links the infinitely extendible subclass of exchangeable narrow-sense geometric distributions to infinitely divisible distributions on $[0,\infty]$.     
          
\begin{definition}[$d$-monotone and $d$-log-monotone sequences] \label{d-mon}
A finite sequence $(x_0,\ldots,x_{d-1})\in\IR^d$ is said to be $d$-monotone if it satisfies
\begin{gather}
\nabla^{j}x_k:=\sum^{j}_{i=0}(-1)^{i}\,\binom{j}{i}\,x_{k+i}\geq0,\quad k=0,1,\ldots,d-1, \; j=0,1,\ldots,d-k-1.
\label{d-mon1}
\end{gather}
An infinite sequence $\{x_k\}_{k\in\IN_0}$ is called completely monotone if $\nabla^{j}x_k\geq0$ for all $k\in\IN_0$, $j\in\IN_0$. 
We define the set $\mathcal{M}_d$ of $d$-monotone and $\mathcal{M}_{\infty}$ of completely monotone sequences, starting with $1$, by
\begin{align*}
\mathcal{M}_d&\ := \big\{\{x_k\}_{k=0}^{d-1}\in\IR^d\,|\,x_0=1,x_1\hspace{-0.1cm}<\hspace{-0.1cm}1,\{x_k\}_{k=0}^{d-1} \text{\;is d-monotone}\big\},\qquad\\
\mathcal{M}_{\infty}&\ :=\big\{\{x_k\}_{k\in\IN_0}\hspace{-0.1cm}\in\IR^{\IN_0}\,|\,x_0=1,x_1\hspace{-0.1cm}<\hspace{-0.1cm}1,\{x_k\}_{k\in\IN_0} \text{\;is completely monotone}\big\},\qquad
\end{align*}
respectively. Further, we define the set $\mathcal{LM}_d$ of $d$-log-monotone and $\mathcal{LM}_{\infty}$ of completely log-monotone sequences, starting with $1$, by
\begin{align*}
\mathcal{LM}_d&\ := \big\{\{x_k\}_{k=0}^{d-1}\hspace{-0.1cm}\in\hspace{-0.1cm}\IR_{>0}^d\,|\,x_0=1,x_1\hspace{-0.1cm}<\hspace{-0.1cm}1,\nabla^{d-k-1}\hspace{-0.1cm}\ln x_k\hspace{-0.1cm}\geq\hspace{-0.1cm}0,\,k\hspace{-0.1cm}=\hspace{-0.1cm}0,1,\ldots,d-2\big\},\qquad\;\;\\
\mathcal{LM}_{\infty}&\ :=\big\{\{x_k\}_{k\in\IN_0}\hspace{-0.1cm}\in\hspace{-0.1cm}\IR_{>0}^{\IN_0}\,|\,x_0=1,x_1<1,\nabla^{j}\ln x_k\geq0,\,k\in\IN_0,\,j\in\IN\big\},\qquad\;\;
\end{align*}
respectively.
Lastly, we define the set $\mathcal{SM}_d$ of $d$-strong-monotone and $\mathcal{SM}_{\infty}$ of completely strong-monotone sequences by
\begin{align*}
\mathcal{SM}_d&\ := \big\{\{x_k\}_{k=0}^{d-1}\in\IR_{>0}^d\,|\,x_0<1,\{\ln x_k^{-1}\}_{k=0}^{d-1} \text{\;is d-monotone}\big\},\qquad\qquad\qquad\;\\
\mathcal{SM}_{\infty}&\ :=\big\{\{x_k\}_{k\in\IN_0}\in\IR_{>0}^{\IN_0}\,|\,x_0<1,\{\ln x_k^{-1}\}_{k\in\IN_0} \text{\;is completely monotone}\big\},\qquad\qquad\qquad\;
\end{align*}
respectively. 
\end{definition}  

For $1\leq j\leq d-1-k$, we have $\nabla^{j}x_k=\nabla^{j-1}x_k-\nabla^{j-1}x_{k+1}$. Further, condition (\ref{d-mon1}) is equivalent to $\nabla^{d-1-k}x_k\geq0$ for $k=0,1,\ldots,d-1$, \red{see \cite{schoenberg32}} or \cite{mai09}. 

By definition, each sequence $\{x_k\}_{k=0}^{d-1}\in\mathcal{M}_d$ (resp.\ $\{x_k\}_{k\in\IN_0}\in\mathcal{M}_{\infty}$) is $[0,1]$-valued and decreasing, whereas each sequence $\{x_k\}_{k=0}^{d-1}\in\mathcal{SM}_d$ (resp.\ $\{x_k\}_{k\in\IN_0}\in\mathcal{SM}_{\infty}$) is $(0,1]$-valued and increasing. Further, in the definition of the set $\mathcal{LM}_d$, $\ln x_{d-1}\geq0$ is the only missing condition for $\{\ln x_k\}_{k=0}^{d-1}$ to be $d$-monotone. Moreover, as we show later in this section, there exists a close relationship between the classes $\mathcal{M}_d$, $\mathcal{LM}_d$, and $\mathcal{SM}_d$, such as, e.g., $\mathcal{LM}_d\subsetneq\mathcal{M}_d$ and $\mathcal{LM}_{\infty}\subsetneq\mathcal{M}_{\infty}$, see Figure \ref{tab_intro}. 

Trivially, each completely monotone sequence $\{x_k\}_{k\in\IN_0}$ is $d$-monotone for all $d\in\IN$, when truncated at $x_{d-1}$. Analogously, each completely log-monotone sequence $\{x_k\}_{k\in\IN_0}$ is $d$-log-monotone for all $d\in\IN$, when truncated at $x_{d-1}$. However, not every $d$-(log-) monotone sequence can be extended to a completely (log-)monotone sequence. Consider the following example. 

\begin{example}[Proper $3$-monotone sequence]\label{mon}
Consider the sequence\\ $(x_0,x_1,x_2):=(1,1/2,1/5)$. It is easy to check that $(x_0,x_1,x_2)\in\mathcal{M}_3$. Assume that $(x_0,x_1,x_2)$ can be extended to an infinite completely monotone sequence $\{x_k\}_{k\in\IN_0}\in\mathcal{M}_{\infty}$. Then, by Hausdorff's moment problem, $x_k=\IE[\tau^k]$ for $k=0,1,2$ and a $[0,1]$-valued random variable $\tau$. Jensen's inequality gives a contradiction:
\begin{gather*}
\frac{1}{5}=x_2=\IE[\tau^2]\geq\IE[\tau]^2=x_1^2=\frac{1}{4}.
\end{gather*}
Indeed, it follows from the Hankel determinants criterion, see \cite{karlin53} and \cite{dette97}, p.\ 20, that only $(1,1/2,x_2)\in\mathcal{M}_3$ with $x_2\geq1/4$ can be extended to an infinite completely monotone sequence. In contrast, each $(1,x_1)\in\mathcal{M}_2$ is extendible to a completely monotone sequence.        
\end{example}

The next theorem constitutes an analytical characterization of $d$-variate exchangeable narrow- and wide-sense geometric distributions in terms of $d$-monotone and $d$-log-monotone sequences, respectively. This facilitates a rigorous study of $\mathcal{G^{N,X}}$ and $\mathcal{G^{W,X}}$. In particular, it helps to determine the subclasses with conditionally i.i.d.\ components in the sense of de Finetti's theorem, see Section \ref{sec:ext}. 
  
\begin{theorem}[Characterization of the exchangeable geometric law] \label{exch}\text{}
\vspace{-0.2cm}
\begin{itemize}
\item[($\mathcal{N}$)] Consider the parametric family $\mathcal{G^{N,X}}$ of all exchangeable, non-degenerate $d$-dimensional narrow-sense geometric survival functions. Then,
\begin{align*}
\mathcal{G^{N,X}}&\ =\big\{\bar{F}^{\mathcal{N}}_{n_1,\ldots,n_d}=\prod_{k=1}^{d}a_k^{n_{(d-k+1)}}\,|\,\{a_k\}_{k=1}^d\in\mathcal{SM}_d\big\}\\
     &\ =\big\{\bar{F}^{\mathcal{N}}_{n_1,\ldots,n_d}=\prod_{k=1}^{d}b_k^{n_{(d-k+1)}-n_{(d-k)}}\,|\,(1,b_1,\ldots,b_d)\in\mathcal{LM}_{d+1}\big\}.
\end{align*} 
The parameters $p_1,\ldots,p_d\in (0,1]$ with $\prod_{i=1}^{d}p_i<1$ and $(a_1,\ldots,a_d)\in\mathcal{SM}_d$ are related via
\begin{align*}
a_k&\ =\prod_{i=1}^{d-k+1}p_i^{\binom{d-k}{i-1}},\quad k=1,\ldots,d,\\
p_k&\ =\prod_{i=1}^k a_{d-i+1}\,^{(-1)^{(k-i)}\binom{k-1}{i-1}}, \quad k=1,\ldots,d.
\end{align*}
The parameters $\{b_k\}_{k=1}^d$ with $(1,b_1,\ldots,b_d)\in\mathcal{LM}_{d+1}$ are given as in (\ref{b's}).   
\item[($\mathcal{W}$)] Consider the parametric family $\mathcal{G^{W,X}}$ of all exchangeable $d$-dimensional wide-sense geometric survival functions. Then,
\begin{gather*}
\mathcal{G^{W,X}}=\big\{\bar{F}^{\mathcal{W}}_{n_1,\ldots,n_d}=\prod_{k=1}^{d}\beta_k^{n_{(d-k+1)}-n_{(d-k)}}\,|\,(1,\beta_1,\ldots,\beta_d)\in\mathcal{M}_{d+1}\big\}.
\end{gather*} 
The parameters $\tilde{p}_1,\ldots,\tilde{p}_d\hspace{-0.1cm}\in\hspace{-0.1cm} [0,1]$ with $\sum_{i=1}^{d}\hspace{-0.1cm}\binom{d}{i-1}\tilde{p}_i\hspace{-0.1cm}\leq\hspace{-0.1cm}1$, $\sum_{i=1}^{d}\hspace{-0.1cm}\binom{d-1}{i-1}\tilde{p}_i\hspace{-0.1cm}<\hspace{-0.1cm}1$, and $\{\beta_k\}_{k=1}^d$ with $(1,\beta_1,\ldots,\beta_d)\in\mathcal{M}_{d+1}$ are related via
\begin{align*}
\beta_k&\ =\sum_{i=1}^{d-k+1}\binom{d-k}{i-1}\tilde{p}_i,\quad k=1,\ldots,d,\\
\tilde{p}_k&\ =\nabla^{k-1}\beta_{d-k+1},\quad k=1,\ldots,d.
\end{align*}
\end{itemize}
\end{theorem}
\begin{proof} Refer to \ref{sec:appb}.\qed
\end{proof}

\begin{example}[Three-variate exchangeable geometric law]\label{mexchgd}\text{}\\
Consider the survival function (\ref{Geo_ind}) of a three-variate distribution with mutually independent geometric marginals from Example \ref{mgd}. In the exchangeable case, i.e. $p_1=p_2=p_3=:p$, (\ref{Geo_ind}) simplifies to
\begin{gather}
\bar{F}_{n_1,n_2,n_3}=p^{n_1+n_2+n_3}=\prod_{k=1}^{3}b_k^{n_{(4-k)}-n_{(3-k)}},
\label{Geo_ind_exch}
\end{gather}
with $b_i=p^i$, $i=1,2,3$. It can be easily verified that $(1,p,p^2,p^3)\in\mathcal{LM}_{4}$.

Another example of an exchangeable three-variate narrow sense geometric distribution is $(\tau_1,\tau_2,\tau_3)$ with identically $\Geo(1-p)$-distributed marginals, $p\in[0,1)$, and the dependence structure induced by the upper Fr\'echet-Hoeffding bound. The survival function of $(\tau_1,\tau_2,\tau_3)$ is then given by 
\begin{gather}
\bar{F}_{n_1,n_2,n_3}=p^{n_{(3)}}=\prod_{k=1}^{3}b_k^{n_{(4-k)}-n_{(3-k)}},
\label{Geo_hf_exch}
\end{gather} 
with $b_i=p$, $i=1,2,3$. Trivially, $(1,p,p,p)\in\mathcal{LM}_{4}$. Note that the lower Fr\'echet-Hoeffding bound with homogeneous geometric margin\-als does not lead to a wide-sense distribution in dimension $d=2$. In the non-exchangeable case, neither the upper nor the lower Fr\'echet-Hoeffding bound with geometric marginals satisfy the local LM property.   
\end{example}

Theorem \ref{exch} $(\mathcal{N})$ provides an analytical characterization of the exchangeable, non-degener\-ate narrow-sense geometric distribution. The only distribution not covered by Theorem \ref{exch} $(\mathcal{N})$ is the degenerate distribution with point mass at $\{1\}^d$. It can be obtained by setting $p_k=0$ for at least one $k=1,\ldots,d$, or, equivalently, $a_1=0$ and $b_k=0$ for all $k=1,\ldots,d$. The corresponding survival function is $\bar{F}^{\mathcal{N}}_{n_1,\ldots,n_d}=\ind_{\{n_1=\ldots=n_d=0\}}$. In the context of a wide-sense geometric distribution, the degenerate case results from $p_k=1$ for all $k=1,\ldots,d$, or, equivalently, $\beta_k=0$ for all $k=1,\ldots,d$. Note also that each component of an exchangeable narrow-sense geometric distribution is $Geo(1-a_1)$-distributed with $$a_1=b_1=\prod_{i=1}^{d}p_i^{\binom{d-1}{i-1}},$$ whereas each component of an exchangeable wide-sense geometric distribution is $Geo(1-\beta_1)$-distributed with $\beta_1=\sum_{i=1}^{d}\binom{d-1}{i-1}\tilde{p}_i$.

In Theorem \ref{exch} $(\mathcal{W})$, the original parameters $\tilde{p}_k$, $k=1,\ldots,d$, of an exchangeable wide-sense geometric distribution are explicitly given in terms of the $(k-1)$-fold difference operator of $\beta_{d-k+1}$. Analogous relationships can be established in the case of an exchangeable narrow-sense geometric distribution as well. Namely, for $k=1,\ldots,d$, it holds that
$$p_k=\exp(-\nabla^{k-1}\ln a^{-1}_{d-k+1})=\exp(-\nabla^k\ln b_{d-k}).$$      
\par
Theorem \ref{prop} shows that the condition $(1,\beta_1,\ldots,\beta_d)\in\mathcal{M}_{d+1}$ in Theorem \ref{exch} $(\mathcal{W})$ is even \emph{necessary and sufficient} for the multi-indexed sequence $\{\bar{F}_{n_1,\ldots,n_d}\}_{n_1,\ldots,n_d \in \IN_0}$ in (\ref{Geo_surv_new_w}) to define a discrete survival function.

\begin{theorem}[Analytical characterization of the survival function]\label{prop}\text{}\\ The multi-indexed sequence $\{\bar{F}_{n_1,\ldots,n_d}\}_{n_1,\ldots,n_d \in \IN_0}$, given by
\begin{gather*}
\bar{F}_{n_1,\ldots,n_d}=\prod_{k=1}^{d}\beta_k^{n_{(d-k+1)}-n_{(d-k)}}, \quad n_1,\ldots,n_d \in \IN_0,
\end{gather*} 
is a d-dimensional survival function if and only if $(1,\beta_1,\ldots,\beta_d)\in\mathcal{M}_{d+1}$. In this case, $\{\bar{F}_{n_1,\ldots,n_d}\}_{n_1,\ldots,n_d \in \IN_0}$ defines an exchangeable wide-sense geometric distribution.
\end{theorem}
\begin{proof} \red{Refer to \ref{sec:appb}. The proof in the published manuscript is overly complicated since the key property of the distribution is not exploited. We have replaced it by a very simple argumentation based on Theorems \ref{Uniq} and \ref{exch}.} \qed
\end{proof}      

The following example illustrates that $\mathcal{G^{N,X}}$-distributions form a proper subclass of $\mathcal{G^{W,X}}$-distributions.     

\begin{example}[Geometric in the wide-, but not in the narrow sense]\label{nw1}
Consider again the sequence $(1,1/2,1/5)\in\mathcal{M}_3$ from Example \ref{mon}. From Theorem \ref{prop}, it follows that 
\begin{gather*}
\bar{F}_{n_1,n_2}=a_1^{\max(n_1,n_2)}a_2^{\min(n_1,n_2)}, \quad n_1,n_2 \in \IN_0,
\end{gather*} 
with $(a_1,a_2)=(1/2,2/5)$, defines an exchangeable bivariate wide-sense geo\-metric distribution. However, by Theorem \ref{exch}, this distribution is not a $\mathcal{G^{N,X}}$-distribution, since 
$$\nabla^1\ln a_1^{-1}=\ln\frac{4}{5}<0,$$
and, consequently, $(a_1,a_2)\notin\mathcal{SM}_2$. Recall that the condition $(a_1,a_2)\in\mathcal{SM}_2$ ensures that the original parameters $p_1$, $p_2$ of the $\mathcal{G^{N,X}}$-distribution are $(0,1]$-valued. The above choice of $(a_1,a_2)$ leads to $p_1=a_2=2/5$ and $p_2=a_1/a_2=5/4>1$.     
\end{example}

Next, we discuss conditions under which a $d$-variate exchangeable distribution can be obtained by truncation of an infinite exchangeable random sequence. By virtue of de Finetti's seminal theorem, this is equivalent to determining the subclasses with conditionally independent and identically distributed components.

\section{Extendible subfamilies}
\label{sec:ext}

In this section, we provide necessary and sufficient conditions on parameters of the survival function in (\ref{Geo_surv_new}) (resp.\ in (\ref{Geo_surv_new_w})) to define a survival function of a narrow- (resp.\ wide-sense) geometric distribution in any dimension $d\in\IN$. Consider the exchangeable random vector $(\tau_1,\ldots,\tau_d)$, geometrically distributed (in the narrow or in the wide sense) with original parameters $p_{1,d},\ldots,p_{d,d}$. Clearly, e.g.\ by Theorem \ref{exch}, every subvector has a geometric law as well. Denote the parameters of a subvector of $(\tau_1,\ldots,\tau_d)$ of length $j$, $j\in\{1,\ldots,d\}$, by  $p_{k,j}$, $k=1,\ldots,j$. For the sake of notational convenience, throughout Section \ref{sec:ext}, we omit the tilde symbol on the parameters $p_{k,j}$ of a wide-sense distribution. We depict the relation between the $p_{k,j}$, $j\in\{1,\ldots,d\}$, $k=1,\ldots,j$, in the following triangular scheme: 

\begin{gather*}
p_{1,1}\\ 
p_{1,2} \qquad p_{2,2}\\
p_{1,3}\qquad p_{2,3}\qquad p_{3,3}\\
\iddots\qquad\qquad\qquad\qquad\qquad\ddots\\
p_{1,d-1}\quad \qquad\qquad\;\;\;\ldots\quad\;\qquad\quad\;\; p_{d-1,d-1}\\
{p_{1,d}}\qquad\;\, p_{2,d}\quad\;\;\qquad\quad\ldots\quad\qquad p_{d-1,d}\qquad\quad p_{d,d}
\end{gather*}

By Remark \ref{repar}, the parameters of two adjacent rows are related multiplicatively, via $p_{k,j}=p_{k,j+1}\,p_{k+1,j+1}$, for all $k=1,\ldots,j$, $j=1,\ldots,d-1$, in the case of a $\mathcal{G^{N,X}}(\textbf{p})$-distribution, and additively, via $p_{k,j}=p_{k,j+1}+p_{k+1,j+1}$, in the case of a $\mathcal{G^{W,X}}(\textbf{p})$-distribution. Recall that in the context of a $\mathcal{G^{N,X}}(\textbf{p})$-distribution, the parameter sequence $(a_1,\ldots,a_j)$, given by
\begin{gather*}
a_k=p_{1,k}=\prod_{i=1}^{d-k+1}p_{i,d}^{\binom{d-k}{i-1}},\qquad k=1,\ldots,j,
\end{gather*}
determines all $k$-dimensional marginal distributions for $k=1,\ldots,j$, and each $\tau_k$ is distributed according to $Geo(1-a_1)$. For a $\mathcal{G^{W,X}}(\textbf{p})$-distribution, all $k$-dimensional marginals are determined by the parameter sequence $(\beta_1,\ldots,\hspace{-0.05cm}\beta_j\hspace{-0.05cm})$, where
\begin{gather*}
\beta_k=p_{1,k}=\sum_{i=1}^{d-k+1}\binom{d-k}{i-1}p_{i,d}, \qquad k=1,\ldots,j,
\end{gather*}
and each $\tau_k$ is distributed according to $Geo(1-\beta_1)$. 

Extendibility of $(\tau_1,\ldots,\tau_d)$ \emph{by one dimension} to an exchangeable geometrically distributed $(\tau_1,\ldots,\tau_d,\tau_{d+1})$ 
is equivalent to extendibility of a given triangle by an additional row $p_{1,d+1},\ldots,p_{d,d+1},p_{d+1,d+1}$ without violating the multiplicative, resp.\ additive, structure. For a $\mathcal{G^{N,X}}(\textbf{p})$-distribution, it is the case precisely when there exists a $q\in(0,1]$ s.t.
\begin{gather*}
q^{-1}\,\prod_{i=1}^{2k-1}p_{i,d}^{(-1)^{i+1}}\in(0,1],\qquad  q\,\prod_{i=1}^{2k}p_{i,d}^{(-1)^{i}}\in(0,1],
\end{gather*}
for all $k=1,\ldots,d/2$, if $d$ is even, or $k=1,\ldots,((d-1)/2+1)$, if $d$ is odd. Similarly, for a $\mathcal{G^{W,X}}(\textbf{p})$-distribution, extendibility by one dimension is equivalent to the existence of a $q\in[0,1]$ s.t.
\begin{gather*}
\sum_{i=1}^{2k-1}(-1)^{i+1}p_{i,d}-q\in[0,1],\qquad \sum_{i=1}^{2k}(-1)^{i}p_{i,d}+q\in[0,1],
\end{gather*}
for all $k=1,\ldots,d/2$, if $d$ is even, or $k=1,\ldots,((d-1)/2+1)$, if $d$ is odd. The extendibility of $(\tau_1,\ldots,\tau_d)$ to an \emph{infinite} exchangeable sequence $\{\tau_k\}_{k\in\IN}$, such that every finite subvector follows an exchangeable geometric distribution, corresponds to infinite extendibility of the associated triangle. This is equivalent to extendibility of the sequence $\{a_k\}_{k=1}^d$ to an infinite completely strong-monotone sequence $\{a_k\}_{k=1}^{\infty}\in\mathcal{SM}_{\infty}$, resp.\ extendibility of the sequence $(1,\beta_1,\ldots,\beta_d)$ to an infinite completely monotone sequence. Corollary \ref{ext} summarizes these results.       

\begin{corollary}[Characterization of the infinitely extendible geometric law] \label{ext}
Denote $b_0:=1$ and $\beta_0:=1$.
\begin{itemize}
\item[($\mathcal{N})$] Consider the parametric family $\mathcal{G^{N,E}}$ of all infinitely extendible, non-dege\-nerate $d$-dimensional narrow-sense geometric survival functions. Then,
\begin{align*}
\hspace{-1.6cm}\mathcal{G^{N,E}}&\ \hspace{-0.2cm}=\hspace{-0.05cm}\big\{\bar{F}^{\mathcal{N}}_{n_1,\ldots,n_d}\hspace{-0.05cm}=\hspace{-0.1cm}\prod_{k=1}^{d}a_k^{n_{(d-k+1)}}|\{a_k\}_{k=1}^d \text{ is extendible to }\{a_k\}_{k\in\IN}\hspace{-0.08cm}\in\hspace{-0.08cm}\mathcal{SM}_{\infty}\big\}\\
     &\ \hspace{-0.2cm}=\hspace{-0.05cm}\big\{\bar{F}^{\mathcal{N}}_{n_1,\ldots,n_d}\hspace{-0.05cm}=\hspace{-0.1cm}\prod_{k=1}^{d}b_k^{n_{(d-k+1)}-n_{(d-k)}}|\{b_k\}_{k=0}^d\text{ is extendible to }\{b_k\}_{k\in\IN_0}\hspace{-0.08cm}\in\hspace{-0.08cm}\mathcal{LM}_{\infty}\big\}.
\end{align*}  
\item[($\mathcal{W}$)] Consider the parametric family $\mathcal{G^{W,E}}$ of all infinitely extendible $d$-dimen\-sional wide-sense geometric survival functions. Then,
\begin{gather*}
\hspace{-0.8cm}\mathcal{G^{W,E}}\hspace{-0.1cm}=\hspace{-0.05cm}\big\{\bar{F}^{\mathcal{W}}_{n_1,\ldots,n_d}\hspace{-0.05cm}=\hspace{-0.1cm}\prod_{k=1}^{d}\beta_k^{n_{(d-k+1)}-n_{(d-k)}}|\{\beta_k\}_{k=0}^d\text{ is extendible to }\{\beta_k\}_{k\in\IN_0}\hspace{-0.08cm}\in\hspace{-0.08cm}\mathcal{M}_{\infty}\big\}.
\end{gather*} 
\end{itemize}
\end{corollary}
\begin{proof} The assertion follows immediately from Theorem \ref{exch}.\qed
\end{proof}

By Hausdorff's moment problem, the condition $\{\beta_k\}_{k\in\IN_0}\in\mathcal{M}_{\infty}$ implies that $\{\beta_k\}_{k\in\IN_0}$ is the sequence of moments of a probability measure on $[0,1]$. In a similar manner, the completely log-monotone sequences $\{\hspace{-0.04cm}b_k\hspace{-0.03cm}\}_{k\in\IN_0}\hspace{-0.14cm}\in\hspace{-0.08cm}\mathcal{LM}_{\infty}$ are related to infinitely divisible probability measures.

\subsection{Relation to infinitely divisible distributions}
\label{sec:ext,infdiv}

The log-monotonicity of an infinite sequence is closely related to the existence of an infinitely divisible probability measure on $[0,\infty]$. Recall that a $[0,\infty]$-valued random variable $X$ is called infinitely divisible if for every $n\geq2$ there exists a sequence of $n$ independent and identically distributed $[0,\infty]$-valued random variables, whose sum equals $X$ in distribution.

\begin{example}[Infinitely divisible distributions on \text{$[0,\infty]$}]  
Every infinitely divisible distribution, discrete or continuous, on $[0,\infty)$ can be modified to an infinitely divisible distribution on $[0,\infty]$, by weighting the probability mass function so that it integrates to some $a$, $a\in(0,1)$, over $[0,\infty)$, and assigning the probability of $1-a$ to the value $\infty$. We illustrate this procedure taking the geometric distribution as an example. To this end, let $X\sim Geo(1-p)$. The distribution of $X$ is infinitely divisible. More precisely, comparing the characteristic function of $X-1$ with the one of the negative binomial distribution (see \cite{sato99}, p.\ 14), one immediately sees that for each $n\geq2$, $X-1$ equals in distribution the sum of $n$ independent and identically negative binomially NB$(\frac{1}{n},p)$-distributed random variables. Now consider a modification $\widetilde{X}$ distributed on $\IN\cup\{\infty\}$ according to $\widetilde{X}\,|\,\widetilde{X}\hspace{-0.1cm}<\hspace{-0.1cm}\infty\sim\text{Geo}(1-p)$, $\IP(\widetilde{X}<\infty)=a$, and $\IP(\widetilde{X}=\infty)=1-a$ for some $a\in(0,1)$. 
Next, we show that $\widetilde{X}$ is infinitely divisible. Fix an arbitrary $n\geq2$. Let $Y_1,\ldots,Y_n$ be independent and identically distributed random variables with $Y_1\,|\,Y_1\hspace{-0.1cm}<\hspace{-0.1cm}\infty\sim\text{NB}(\frac{1}{n},p)$, $\IP(Y_1<\infty)=a^{\frac{1}{n}}$, and $\IP(Y_1=\infty)=1-a^{\frac{1}{n}}$. Then, 
$$\IP\Big(\hspace{-0.05cm}1\hspace{-0.05cm}+\hspace{-0.07cm}\sum_{i=1}^n \hspace{-0.05cm}Y_i\hspace{-0.1cm}=\hspace{-0.05cm}k\hspace{-0.05cm}\Big)\hspace{-0.1cm}=\hspace{-0.05cm}\IP\Big(\hspace{-0.1cm}\sum_{i=1}^n \hspace{-0.05cm}Y_i\hspace{-0.1cm}=\hspace{-0.05cm}k-1\big|Y_i<\infty, i\hspace{-0.1cm}=\hspace{-0.1cm}1,\ldots,n\Big)\IP(Y_1\hspace{-0.1cm}<\hspace{-0.1cm}\infty)^n\hspace{-0.1cm}=\hspace{-0.05cm}\IP(\widetilde{X}\hspace{-0.1cm}=\hspace{-0.05cm}k),$$ and $$\IP\Big(1+\sum_{i=1}^n Y_i=\infty\Big)=1-\big(1-\IP(Y_1=\infty)\big)^n=1-a=\IP(\widetilde{X}=\infty).$$
Thus, $\widetilde{X}$ equals in distribution $\sum_{i=1}^n(Y_i+1/n)$, and the assertion follows. 
\end{example}

\red{The next lemma was originally used in the proof of Proposition \ref{extend} and is left here for the sake of our erratum.}
\begin{lemma}[\red{Gnedin and Pitman \cite{gnedin08}}] \label{moments}
Let $X$ be a $[0,\infty]$-valued random variable with $\IP(X<\infty)>0$. Then, the following statements are equivalent.
\begin{enumerate}
	\item[(i)] The distribution of $X$ is infinitely divisible on $[0,\infty]$. 
	\item[(ii)] For all $k\in\IN_0$, it holds that $\IE[e^{- k\,X}]=e^{-\mu_k}$, where $\{-\nabla\mu_k\}^{\infty}_{k=0}$ is completely monotone and $\mu_0=0$.
\end{enumerate}
\end{lemma}
\begin{proof} \red{This is simply a reformulation of Corollary $4.2$ in \cite{gnedin08} using their Definition $1.5$. In the published manuscript, a nonsensical proof was given.}\qed
\end{proof}

\red{Proposition \ref{extend} is a simple corollary of Hausdorff's moment problem. It} establishes a relationship between completely monotone sequences, completely log-monotone sequences, and the existence of an infinitely divisible probability measure.  

\begin{proposition}[\red{Gnedin and Pitman \cite{gnedin08}}]\label{extend}
\red{Let $\{b_k\}_{k\in\IN_0}$ be sequence of strictly positive real numbers.} The following statements are equivalent.
\begin{enumerate}
	\item[(i)] $\{b_k\}_{k\in\IN_0}\in\mathcal{LM}_{\infty}$.
	\item[(ii)] $\{b_k\}_{k\in\IN_0}$ is given by
    \begin{gather*}
    b_k=\int_{[0,\infty]}e^{-kx}\,\kappa(dx),\quad k\in\IN_0,
    \end{gather*}
for an infinitely divisible probability measure $\kappa$ on $[0,\hspace{-0.025cm}\infty]$ with \red{$\kappa(\{0\})<1$ and $\kappa(\{\infty\})<1$}.	
	\item[(iii)] $\{b_k^r\}^{\infty}_{k=0}\in\mathcal{M}_{\infty}$ for each fixed $r>0$.
	\vspace{0.07cm}
	\item[(iv)] $\{\sqrt[n]{b_k}\;\}^{\infty}_{k=0}\in\mathcal{M}_{\infty}$ for each fixed $n\in\IN$.
\end{enumerate}
\end{proposition}
\begin{proof} \red{The equivalence of $(i)$ and $(ii)$ is simply a reformulation of Lemma~\ref{moments} (in a context where we restrict ourselves to $b_1<1$, or, equivalently, to $\kappa$ being not concentrated at $0$). In the published manuscript, a nonsensical proof relying on the results of \cite{lorch83} was given. In \ref{sec:appc}, we provide an alternative proof to the one given in \cite{gnedin08}.}\qed
\end{proof}

In part $(ii)$ of Proposition \ref{extend}, we use the convention $\infty\cdot0=0$. The equivalence of $(i)$ and $(ii)$ constitutes an analogue of Hausdorff's moment problem. Results analogous to the equivalences of $(ii)$, $(iii)$, and $(iv)$ for the Hausdorff moment sequence $\{b_k\}^{\infty}_{k=0}\in\mathcal{M}_{\infty}$ have been obtained for normalized Stieltjes as well as normalized Hamburger moment sequences. In our study, the complete log-monotonicity arises naturally as a characterizing property of parameter sequences of the $\mathcal{G^{N,E}}$-law. Interestingly, it relates to the existing theory on logarithmically completely monotonic functions on $(0,\infty)$, for which, characterizations by completely monotonic functions, positive definite functions, or moment sequences resembling Proposition \ref{extend} are known. For an overview of the existing results on logarithmically completely monotonic functions and moment sequences, we refer to \cite{berg07,berg08}.       
\par
As an immediate consequence to the finite-dimensional exchangeable case, the next corollary states that $(1,\beta_1,\beta_2,\ldots)\in\mathcal{M}_{\infty}$ is \emph{necessary and sufficient} for (\ref{Geo_surv_new_w}) to define a survival function in any dimension $d\in\IN$.

\begin{corollary}[Analytical characterization of the survival function] \label{analytic}\text{ }
The multi-indexed sequence $\{\bar{F}_{n_1,\ldots,n_d}\}_{n_1,\ldots,n_d \in \IN_0}$, given by
\begin{gather*}
\bar{F}_{n_1,\ldots,n_d}=\prod_{k=1}^{d}\beta_k^{n_{(d-k+1)}-n_{(d-k)}}, \quad n_1,\ldots,n_d \in \IN_0,
\end{gather*} 
is a d-dimensional survival function for each $d\geq2$ if and only if $(1,\beta_1,\beta_2,\ldots)\in\mathcal{M}_{\infty}$. In this case, $\{\bar{F}_{n_1,\ldots,n_d}\}_{n_1,\ldots,n_d \in \IN_0}$ defines an infinitely extendible wide-sense geometric distribution.
\end{corollary}
\begin{proof} The corollary follows from Theorem \ref{prop}.\qed
\end{proof}

In view of de Finetti's theorem, extendible distributions are characterized by conditionally i.i.d.\ components. A stochastic representation in the sense of de Finetti's theorem can be achieved for both the $\mathcal{G^{N,E}}$ and $\mathcal{G^{W,E}}$ families.

\subsection{Stochastic representation in the sense of de Finetti's theorem}
\label{sec:ext,ciid}

Theorem \ref{thm_ciid} provides an alternative stochastic model with components that are i.i.d.\ conditioned on a non-decreasing random walk for both the $\mathcal{G^{N,E}}$ and $\mathcal{G^{W,E}}$ families. The $\mathcal{G^{N,E}}$-distribution corresponds to the proper subclass constructed by random walks with infinitely divisible increments. 

\begin{theorem}[\red{Esary and Marshall \cite{esary73}}] \label{thm_ciid}
Consider a probability space $(\Omega,\F,\IP)$ supporting a sequence $\{X_k\}_{k\in \IN}$ of i.i.d.\ $[0,\infty]$-valued random variables with $\IP(0<X_1\leq\infty)>0$. Furthermore, let $\{E_k\}_{k\in \IN}$ be an i.i.d.\ sequence of $Exp(1)$-distributed random variables, independent of $\{X_k\}_{k\in \IN}$. Define the random variables $\{\tau_k\}_{k \in \IN}$ via
\begin{gather*}
\tau_k:=\min\{n \in \IN\,:\,X_1+\ldots+X_n \geq E_k\},\quad k\in \IN.
\end{gather*}
\begin{enumerate}
\item[($\mathcal{W})$] For each $d \geq 2$, the random vector $(\tau_1,\ldots,\tau_d)$ has an exchangeable $d$-variate wide-sense geometric distribution with survival function (\ref{Geo_surv_new_w}), where
\begin{gather*}
\beta_k=\IE\big[e^{-k\,X_1}\big],\quad k\in \IN.
\end{gather*} 
\item[($\mathcal{N})$] If, additionally, the distribution of $X_1$ is infinitely divisible, then and only then, for each $d \geq 2$, the random vector $(\tau_1,\ldots,\tau_d)$ has an exchangeable narrow-sense geometric distribution with survival function (\ref{kurzi}), where
\begin{gather*}
b_k=\IE\big[e^{-k\,X_1}\big],\quad k\in \IN.
\end{gather*}   
\end{enumerate}
Moreover, every extendible wide-sense (resp.\ narrow-sense) geometric distribution can be constructed using the probabilistic model in $(\mathcal{W})$ (resp.\ in $(\mathcal{N})$). 
\end{theorem}
\begin{proof} Refer to \ref{sec:appc}. \red{While we provide an alternative proof, the results found in \cite{esary73} are more general.}\qed
\end{proof}

In the situation of Theorem \ref{thm_ciid}, one says that the geometrically distributed $(\tau_1,\ldots,\tau_d)$ is \emph{extendible} to an infinite exchangeable sequence. By virtue of de Finetti's theorem, see \cite{finetti37} and \cite{aldous85}, p.\ 20 ff., the random variables $\{\tau_k\}_{k\in\IN}$ are i.i.d.\ conditioned on the $\sigma$-algebra generated by the sequence $\{X_k\}_{k \in \IN}$. On the one hand, an extendible geometric law is parameterized by the sequence of moments of a single $[0,1]$-valued random variable. On the other hand, a single (infinitely divisible) random variable is sufficient for a simple and efficient construction of an extendible wide-sense (resp. narrow-sense) geometric law. This constitutes a practically convenient approach to provide low-parametric families of high-dimensional geometric distributions.

\begin{example}[Three-variate extendible geometric law]\label{mextgd}\text{}\\
Consider the survival function (\ref{Geo_ind_exch}) of a three-variate exchangeable distribution with mutually independent geometric marginals parameterized by the sequence $(1,p,p^2,p^3)$ from Example \ref{mexchgd}. (\ref{Geo_ind_exch}) is infinitely extendible and can be constructed as in Theorem \ref{thm_ciid} $(\mathcal{N})$ using the degenerate distribution $X_1:=-\ln p$. A degenerate distribution is, trivially, infinitely divisible.

The survival function (\ref{Geo_hf_exch}), parameterized by $(1,p,p,p)$ and corresponding to the upper Fr\'echet-Hoeffding bound from Example \ref{mexchgd}, is infinitely extendible as well. It can be constructed as in Theorem \ref{thm_ciid} $(\mathcal{N})$ using the infinitely divisible distribution $\IP(X_1=0)=p=1-\IP(X_1=\infty)$.
\end{example}

\begin{remark}[Link to the Bernoulli sieve]
The construction of Theorem \ref{thm_ciid} involves a non-decreasing random walk with increments $\{X_k\}_{k \in \IN}$. However, there is another possible construction which is motivated by coin tosses. It is called the \emph{Bernoulli sieve} in \cite{gnedin04} and defines $\{\tau_k\}_{k \in \IN}$ as follows: 
\begin{gather*}
\tau_k:=\min\{n \in \IN\,:\,Z^{(n)}_k=1\},\quad k \in \IN,
\end{gather*}
where the sequences $\{Z^{(1)}_k\}_{k \in \IN},\,\{Z^{(2)}_k\}_{k \in \IN},\ldots$ are independent and for each $n \in \IN$, $\{Z^{(n)}_k\}_{k \in \IN}$ is a sequence of Bernoulli random variables conditioned on a random success probability $1-Y_n \in [0,1]$. It is assumed that the sequence $\{Y_n\}_{n \in \IN}$ is i.i.d.\ and $\IP(Y_1<1)>0$. The distribution of the $\{Y_k\}_{k \in \IN}$ is precisely $\eta$ from the proof of Theorem \ref{thm_ciid} above. In words, imagine a game with infinitely many players labeled by $k \in \IN$. The game has (possibly) infinitely many rounds, and in each round a certain number of players drops out. In round $n \in \IN$, a referee randomly picks a coin with probability of tails $1-Y_n$. Then, for each player the coin is thrown once and the player drops out if she gets tails. The random variable $\tau_k$ denotes the drop out round of player $k$. The condition $\IP(Y_1<1)>0$ guarantees that each player drops out eventually.
\par
Whereas the present article studies the distribution of $(\tau_1,\ldots,\tau_d)$, \cite{gnedin04} studies the distribution of $(P_1,\ldots,P_{K_d})$, where 
\begin{gather*}
K_d\hspace{-0.07cm}:=\hspace{-0.07cm}|\{\tau_1,\ldots,\tau_d\}|\hspace{-0.07cm}\hspace{-0.025cm}=\hspace{-0.07cm}\#\mbox{ of rounds in which one of the first }d \mbox{ players drops out},
\end{gather*}
and, with $\tau_{(1)}\leq\ldots\leq\tau_{(d)}$ denoting the order statistics of $(\tau_1,\ldots,\tau_d)$,
\begin{align*}
P_1 &\hspace{-0.07cm}:= \hspace{-0.07cm}\max\{j\hspace{-0.05cm}:\,\tau_{(j)}\hspace{-0.07cm}=\hspace{-0.07cm}\tau_{(1)}\}\hspace{-0.07cm}=\hspace{-0.07cm}\#\mbox{ of drop outs in the first of these rounds},\\
P_2 &\hspace{-0.07cm}:=\hspace{-0.07cm} \max\{j\hspace{-0.05cm}:\,\tau_{(P_1+j)}\hspace{-0.07cm}=\hspace{-0.07cm}\tau_{(P_1+1)}\}\hspace{-0.07cm}=\hspace{-0.07cm}\#\mbox{ of drop outs in the second of these rounds},\\
\vdots\\
P_{K_d} &\hspace{-0.07cm}:= \hspace{-0.07cm}\max\{j\hspace{-0.05cm}:\,\tau_{(P_{K_d-1}+j)}\hspace{-0.07cm}=\hspace{-0.07cm}\tau_{(P_{K_d-1}+1)}\}\hspace{-0.07cm}=\hspace{-0.07cm}\#\mbox{ of drop outs in the last of these rounds}.
\end{align*}
Notice that $P_1,\ldots,P_{K_d} \in \IN$ and $P_1+\ldots+P_{K_d}=d$, hence $(P_1,\ldots,P_{K_d})$ defines a \emph{random partition} of $d$, and vast literature can be found on the study of such objects. \cite{gnedin05} extend the Bernoulli sieve to a ``continuous'' case using L\'evy subordinators. Considering non-decreasing random walks as discrete equivalents of L\'evy subordinators explains their appearance in Theorem \ref{thm_ciid}. Furthermore, \cite{mai11} show that the extendible subclass of Marshall--Olkin exponential distributions can be characterized by L\'evy subordinators. In view of these references, the present article completes the picture, see Table \ref{tab:lim_mem}.   
\end{remark}

\vspace{-0.5cm}
\begin{table}[h]
\caption{Stochastic objects with limited memory.} 
\vspace{0.25cm}
\centering 
\begin{tabular}{c|c|c|c} 
\text{} & \small{\emph{non-decreasing}} & \small{\emph{random integer}} & \small{\emph{extendible}}\\
\text{} & \small{\emph{processes}} & \small{\emph{partitions}} & \small{\emph{multivariate laws}}\\ 
\hline
\text{} & \text{} & \text{} & \text{}\\  
\hspace{-0.2cm}\multirow{3}{*}{\small{\emph{continuous}}} & \small{\emph{L\'evy}} & \small{\emph{regenerative composition}} & \small{\emph{Marshall-Olkin}}\\
\text{  } & \small{\emph{subordinators}} & \small{\emph{structures, see}} & \small{\emph{distributions, see}}\\
\text{  } & \small{\emph{}} & \small\emph{{\cite{gnedin05}}} & \small{\emph{\cite{mai11}}}\\
\hline
\text{} & \text{} & \text{} & \text{}\\  
\hspace{-0.2cm}\multirow{3}{*}{\small{\emph{discrete}}} & \small{\emph{random walks}} & \small{\emph{the Bernoulli sieve,}} & \small{\emph{$\mathcal{G^{W,E}}$ and $\mathcal{G^{N,E}}$,}}\\
\text{} & \small{\emph{with non-negative}} & \small{\emph{see \cite{gnedin04}}} & \small{\emph{the present article}} \\
\text{} & \small{\emph{increments}} & \text{} & \text{} \\
\end{tabular} 
\label{tab:lim_mem}
\end{table}

\section{Dependence structure}
\label{sec:dep}

In this section, we study the dependence structure of the wide- and the narrow-sense geometric distributions discussed in Sections \ref{sec:mgd}, \ref{sec:exch}, and \ref{sec:ext}. Whereas the narrow-sense construction always leads to distributions with non-negatively correlated components, the more general wide-sense geometric distribution allows for negative correlations. Taking into account the specific form of the survival function, the explicit formulas for the correlation coefficients can be conveniently obtained by writing $\tau_i=\sum_{n=0}^{\infty}\ind_{\{\tau_i>n\}}$ for $\IN$-valued $\tau_i$, $i=1,2$, and computing the mixed second moment via
\begin{gather*}
\IE[\tau_1\tau_2]=\sum_{i\geq 0}\sum_{j\geq 0}\IP(\tau_1>i,\tau_2>j).
\end{gather*}           
The computation is tedious but straightforward, and therefore omitted. The formulas for the correlation coefficients for the wide- and the narrow-sense geometric distribution in dimension $d\geq2$, in both exchangeable and general cases, are provided in Lemma \ref{l:corr}. The conditions on the involved parameters are given in Theorem \ref{exch}, in the case of an exchangeable distribution, and in Definition \ref{def}, in the general case.

\begin{lemma}[Correlation coefficients]\label{l:corr}\text{} 
\vspace{-0.2cm}
\begin{enumerate}
\item[($\mathcal{N}$)] Let $d\geq2$ and let $(\tau_1,\ldots,\tau_d)\sim\mathcal{G^N}(\textbf{p})$. Then, for $n,m\in\{1,\ldots,d\}$, $n\neq m$, the correlation coefficient $Corr[\tau_n,\tau_m]$ is given by
$$Corr[\tau_n,\tau_m]=\frac{\sqrt{\prod_{I:(n\in I,m\notin I)\cup(n\notin I,m\in I)} p_I}\big(1-\prod_{I:n,m\in I} p_I\big)}{1-\prod_{I:(n\in I)\cup(m\in I)} p_I}.$$
If, additionally, the distribution of $(\tau_1,\ldots,\tau_d)$ is exchangeable, then 
$$Corr[\tau_n,\tau_m]=\frac{a_2-a_1}{1-a_1a_2}=\frac{b_2-b_1^2}{b_1(1-b_2)},\quad n,m\in\{1,\ldots,d\},\;n\neq m,$$
with $a_1$, $a_2$ and $b_1$, $b_2$ defined in (\ref{ak}) and (\ref{b's}), respectively.	
\item[($\mathcal{W}$)] Let $d\geq2$ and let $(\tau_1,\ldots,\tau_d)\sim\mathcal{G^W}(\tilde{\textbf{p}})$. Then, for $n,m\in\{1,\ldots,d\}$, $n\neq m$, such that $\sum_{I:n\notin I}\tilde{p}_I\neq0$ and $\sum_{I:m\notin I}\tilde{p}_I\neq0$, the correlation coefficient $Corr[\tau_n,\tau_m]$ is given by
$$Corr[\tau_n,\tau_m]=\frac{\sum_{I:n,m\notin I}\tilde{p}_I-\sum_{I:n\notin I}\sum_{J:m\notin J}\tilde{p}_I\tilde{p}_J}{\big(1-\sum_{I:n,m\notin I}\tilde{p}_I\big)\sqrt{\sum_{I:n\notin I}\sum_{J:m\notin J}\tilde{p}_I\tilde{p}_J}}.$$
If, additionally, the distribution of $(\tau_1,\ldots,\tau_d)$ is exchangeable, then 
$$Corr[\tau_n,\tau_m]=\frac{\beta_2-\beta_1^2}{\beta_1(1-\beta_2)},\quad n,m\in\{1,\ldots,d\},\;n\neq m,$$
with $\beta_1$, $\beta_2$ defined in (\ref{beta's}).	
\end{enumerate}
\end{lemma}
\begin{proof} A tedious but straightforward computation provides the results.\qed
\end{proof}

In part $(\mathcal{W})$ of Lemma \ref{l:corr}, we exclude the degenerate case $\sum_{I:n\notin I}\tilde{p}_I=0$, $n\in\{1,\ldots,d\}$, corresponding to $\tau_n=1$ almost surely. For exchangeable $(\tau_1,\ldots,\tau_d)$, this translates into $\beta_1\neq0$.  

\begin{example}[Geometric in the wide-, but not in the narrow sense]\label{ex:nw2}    
Consider again the sequence $(\beta_0,\beta_1,\beta_2)=(1,\frac{1}{2},\frac{1}{5})\in\mathcal{M}_3$ from Examples \ref{mon} and \ref{nw1}, defining parameters of a bivariate $\mathcal{G^{W,X}}$-distribution. With Lemma \ref{l:corr} $(\mathcal{W})$, we compute
\begin{gather*}
Corr[\tau_1,\tau_2]=\frac{\beta_2-\beta_1^2}{\beta_1(1-\beta_2)}=-\frac{1}{8}.
\end{gather*} 
\end{example}

In the case of a wide-sense geometric distribution, the correlation coefficient between every two components is bounded from below by $-0.5$ in any dimension $d\geq2$. The correlation coefficient $Corr[\tau_n,\tau_m]=-0.5$ is obtained in the case $\tilde{p}_I=2^{1-d}$ for all $I\subseteq\{1,\ldots,d\}$ s.t.\ $n\in I,\,m\notin I$ or $m\in I,\,n\notin I$, and $\tilde{p}_I=0$, otherwise. Note that due to $|\{I:n\in I,\,m\notin I\}|=2^{d-2}$, these parameters satisfy the conditions of Definition \ref{def} $(\mathcal{W})$. 

In the case of an exchangeable $d$-variate distribution, the correlation coefficient is bounded from below by $-1/d$ for any $d\geq2$. This can be verified by observing that for $\beta_1\neq0$, the correlation coefficient 
$$Corr[\tau_n,\tau_m]=\frac{\beta_2-\beta_1^2}{\beta_1(1-\beta_2)},\quad n,m=1,\ldots,d,$$
is strictly increasing in $\beta_2$ and strictly decreasing in $\beta_1$. Setting $\beta_2=0$, the condition $(1,\beta_1,\ldots,\beta_{d})\in\mathcal{M}_{d+1}$ gives $\beta_k=0$ for all $k=2,\ldots,d$ and $(1-d\beta_1)\geq0$. Hence, the minimum is attained at $(\beta_0,\ldots,\beta_{d})=(1,1/d,0,\ldots,0)\in\mathcal{M}_{d+1}$. Recall that the lower bound for the correlation coefficient between the components of an arbitrary exchangeable $d$-variate distribution is given by $-1/(d-1)$ for any $d\geq2$, see \cite{kingman78} or \cite{aldous85}, p.\ 8.

\begin{remark}[Correlations induced by the Fr\'echet-Hoeffding bounds]
Let $\tau_1\sim Geo(1-p_1)$ and $\tau_2\sim Geo(1-p_2)$, $p_1,p_2\in(0,1)$. Assume that the dependence between $\tau_1$ and $\tau_2$ is now given by the lower Fr\'echet-Hoeffding bound, i.e.
$$\IP(\tau_1\leq n_1,\tau_2\leq n_2)=\max\Big\{0,1-p_1^{n_1}-p_2^{n_2}\Big\}.$$
Then, the correlation coefficient $Corr[\tau_1,\tau_2]$ is given by
\begin{gather*}
Corr[\tau_1,\tau_2]=\frac{(1\hspace{-0.07cm}-\hspace{-0.07cm}p_1)(1\hspace{-0.07cm}-\hspace{-0.07cm}p_2)}{\sqrt{p_1p_2}}\Big(\sum\limits_{i\geq 0}\sum\limits_{j\geq0}\max\big\{0,p_1^{i}\hspace{-0.07cm}+\hspace{-0.07cm}p_2^{j}\hspace{-0.07cm}-\hspace{-0.07cm}1\big\}\hspace{-0.07cm}-\hspace{-0.07cm}\frac{1}{(1\hspace{-0.07cm}-\hspace{-0.07cm}p_1)(1\hspace{-0.07cm}-\hspace{-0.07cm}p_2)}\Big).
\end{gather*}
In the special case $p_1+p_2\leq1$, the expression for the correlation coefficient simplifies to
\begin{gather*}
Corr[\tau_n,\tau_m]=-\sqrt{p_1p_2}.
\end{gather*}
Recall that correlations induced by the lower Fr\'echet-Hoeffding bound are always non-positive. In our case of geometric marginals, we observe $Corr[\tau_1,\tau_2]\rightarrow 0$ as $p_1\rightarrow0$, $p_2\rightarrow0$, and $Corr[\tau_1,\tau_2]\rightarrow -1$ as $p_1\rightarrow1$, $p_2\rightarrow1$. It is worth noticing that a distribution with non-degenerate geometric marginals and the lower Fr\'echet-Hoeffding copula is not geometric in the wide-sense by Theorem \ref{Uniq}.

Now let the dependence between $\tau_k\sim Geo(1-p_k)$, $p_k\in(0,1)$, $k\in\{1,\ldots,d\}$, be given by the upper Fr\'echet-Hoeffding bound, i.e.
$$\IP(\tau_1\leq n_1,\ldots,\tau_d\leq n_d)=\min\Big\{1-p_1^{n_1},\ldots,1-p_d^{n_d}\Big\}.$$
Then, for $n,m\in\{1,\ldots,d\}$, $n\neq m$, the correlation coefficient $Corr[\tau_n,\tau_m]$ is given by
\begin{gather*}
Corr[\tau_n,\tau_m]=\frac{(1\hspace{-0.07cm}-\hspace{-0.07cm}p_n)(1\hspace{-0.07cm}-\hspace{-0.07cm}p_m)}{\sqrt{p_n p_m}}\Big(\sum\limits_{i\geq 0}\sum\limits_{j\geq0}\min\big\{p_n^{i},p_m^{j}\big\}\hspace{-0.07cm}-\hspace{-0.07cm}\frac{1}{(1\hspace{-0.07cm}-\hspace{-0.07cm}p_n)(1\hspace{-0.07cm}-\hspace{-0.07cm}p_m)}\Big).
\end{gather*}
Recall that correlations induced by the upper Fr\'echet-Hoeffding bound are always non-negative. In the homogeneous case $p_1=p_k$ for all $k=2,\ldots,d$, we obtain $Corr[\tau_n,\tau_m]=1$ for all $p_1\in(0,1)$. More precisely,  
\begin{gather*}
\IE[\tau_n\tau_m] =\sum_{i\geq 0}\sum_{j\geq0}\min\{p_1^i,p_1^j\}=\sum_{i\geq 0}\Big(\sum_{j=0}^ip_1^i+\sum_{j=i+1}^{\infty}p_1^j\Big)=\frac{1+p_1}{(1-p_1)^2},  
\end{gather*}
and, hence, 
\begin{gather*}
Corr[\tau_n,\tau_m]=\Big(\frac{1+p_1}{(1-p_1)^2}-\frac{1}{(1-p_1)^2}\Big)\Big(\frac{(1-p_1)^2}{p_1}\Big)=1.
\end{gather*}
Moreover, in this case, the corresponding distribution is an exchangeable narrow-sense geometric distribution with survival function 
\begin{gather*}
\bar{F}_{n_1,\ldots,n_d}=1+\sum_{i=1}^d\sum_{\stackrel{\emptyset\neq I\subseteq\{1,\ldots,d\}}{|I|=i}}(-1)^i\IP(\tau_j\leq n_j,\,j\in I)=p_1^{n_{(d)}}.
\end{gather*}
\end{remark}

The following example shows that the set of all bivariate $\mathcal{G^{N}}$-distributions is precisely the set of all bivariate $\mathcal{G^{W}}$-distributions with non-negative correlations. 

\begin{example}[Non-negative correlations $\Rightarrow\mathcal{G^{N}}$ in $d=2$]\label{ex:corr2}
In view of Theorems \ref{thm_a} and \ref{Uniq}, every $(\tau_1,\tau_2)\sim\mathcal{G^{N}}$ follows a $\mathcal{G^{W}}$-distribution with non-negatively correlated components. We show that the converse also holds. To this end, let $(\tau_1,\tau_2)\sim\mathcal{G^{W}}$ with $Corr[\tau_1,\tau_2]\geq0$. By Theorem \ref{thm_a} $(\mathcal{W})$, 
\begin{gather*}
\IP(\tau_1>i,\tau_2>j)=\begin{cases}
\big(\tilde{p}_{\emptyset}+\tilde{p}_{\{2\}}\big)^{i-j}\,\tilde{p}_{\emptyset}^j, & \text{ if $i\geq j$,}\\
\big(\tilde{p}_{\emptyset}+\tilde{p}_{\{1\}}\big)^{j-i}\,\tilde{p}_{\emptyset}^i, & \text{ if $i<j$,}
\end{cases}
\end{gather*} 
with parameters satisfying $\sum_{I\subseteq\{1,2\}}\tilde{p}_I=1$, $\sum_{I:k\notin I}\tilde{p}_I<1$, for $k=1,2$, and $\tilde{p}_{\emptyset}-\big(\tilde{p}_{\emptyset}+\tilde{p}_{\{1\}}\big)\big(\tilde{p}_{\emptyset}+\tilde{p}_{\{2\}}\big)\geq0$, where the last condition is due to non-negativity of the correlation coefficient, see Lemma \ref{l:corr} $(\mathcal{W})$. For $\tilde{p}_{\emptyset}>0$, define
\begin{gather*}
p_{\{1\}}:=\frac{\tilde{p}_{\emptyset}}{\tilde{p}_{\emptyset}+\tilde{p}_{\{1\}}},\quad p_{\{2\}}:=\frac{\tilde{p}_{\emptyset}}{\tilde{p}_{\emptyset}+\tilde{p}_{\{2\}}},\quad p_{\{1,2\}}:=\frac{\big(\tilde{p}_{\emptyset}+\tilde{p}_{\{1\}}\big)\big(\tilde{p}_{\emptyset}+\tilde{p}_{\{2\}}\big)}{\tilde{p}_{\emptyset}}.
\end{gather*}  
$p_{\{1\}}$, $p_{\{2\}}$, and $p_{\{1,2\}}$ are well-defined and satisfy the conditions of Definition \ref{def} $(\mathcal{N})$. Further,
\begin{gather*}
\IP(\tau_1>i,\tau_2>j)=\begin{cases}
\big(p_{\{1\}}p_{\{1,2\}}\big)^{i}\,p_{\{2\}}^j, & \text{ if $i\geq j$,}\\
\big(p_{\{2\}}p_{\{1,2\}}\big)^{j}\,p_{\{1\}}^i, & \text{ if $i<j$.}
\end{cases}
\end{gather*} 
Hence, by Theorem \ref{thm_a} $(\mathcal{N})$, $(\tau_1,\tau_2)\sim\mathcal{G^{N}}(p_{\{1\}},p_{\{2\}},p_{\{1,2\}})$.
\end{example}
 
On the contrary, in dimension $d\geq 3$, the set of all $d$-variate $\mathcal{G^{N}}$-distributions is a proper subset of all $d$-variate $\mathcal{G^{W}}$-distributions with non-negative correlations.

\begin{example}[Non-negative correlations $\nRightarrow\mathcal{G^{N}}$ in $d\geq3$]\label{ex:corr3}
Let $d\geq3$ and let $X$ be Bernoulli distributed with success probability $1-p=0.9$. For $k=0,\ldots,d$, define $\beta_k:=\IE[e^{-k X}]=p+(1-p)e^{-k}$. By Theorem \ref{thm_ciid} $(\mathcal{W})$, $(\beta_0,\ldots,\beta_{d})$ defines parameters of a $d$-variate $\mathcal{G^{W,E}}$-distribution. Using the formula for correlation coefficients from Lemma \ref{l:corr} $(\mathcal{W})$, 
\begin{gather*}
Corr[\tau_n,\tau_m]=\frac{\beta_2-\beta_1^2}{\beta_1(1-\beta_2)}=\frac{0.1+0.9e^{-2}-(0.1+0.9e^{-1})^2}{(0.1+0.9e^{-1})(0.9-0.9e^{-2})}\approx0.1072>0, 
\end{gather*} 
for $n,m\in\{1,\ldots,d\}$, $n\neq m$. However, since 
$$\nabla^3\ln\beta_0=\ln\frac{\beta_2^3}{\beta_1^3\beta_3}\approx-0.06<0,$$ 
it follows that $(\beta_0,\ldots,\beta_{d})\notin\mathcal{LM}_{d+1}$, and, hence, in view of Theorem \ref{exch} $(\mathcal{N})$, $(\beta_0,\ldots,\beta_{d})$ does not define parameters of a $d$-variate $\mathcal{G^{N,X}}$-distribution. In this context, note that the Bernoulli distribution is not infinitely divisible, and, by Theorem \ref{thm_ciid} $(\mathcal{N})$, $(\beta_0,\ldots,\beta_{d})$ defines parameters of a $d$-variate $\mathcal{G^{N,X}}$-distribution for each $d\geq2$ if and only if the distribution of $X$ is infinitely divisible.
\end{example}

\begin{remark}[Relationships between subfamilies of the geometric law]
The relationships between exchangeable subfamilies of the multivariate geometric law with limited memory can be summarized as follows:
\begin{enumerate}
\item [(i)] $\mathcal{G^{N,E}}\subseteq\mathcal{G^{N,X}}\subsetneq\mathcal{G^{W,X}}\;\;\text{and}\;\;\mathcal{G^{N,E}}\subseteq\mathcal{G^{W,E}}\subsetneq\mathcal{G^{W,X}}$.\\
The inclusion $\mathcal{G^{N,X}}\subsetneq\mathcal{G^{W,X}}$ is proper in every dimension $d\geq2$ by Example \ref{nw1}. The inclusion $\mathcal{G^{W,E}}\subsetneq\mathcal{G^{W,X}}$ is proper in every dimension $d\geq2$ by Example \ref{mon}.
\item [(ii)] In addition, for exchangeable bivariate geometric distributions:
$$\mathcal{G^{N,X}}=\mathcal{G^{N,E}}=\mathcal{G^{W,E}}\subsetneq\mathcal{G^{W,X}}.$$
More precisely, by Hausdorff's moment problem, $(1,\beta_1,\beta_2)$ is extendible to an infinite sequence $(1,\beta_1,\beta_2,\ldots)\in\mathcal{M}_{\infty}$ if and only if $\beta_k=\IE[\tau^k]$, $k=1,2$, for a $[0,1]$-valued random variable $\tau$. Then, by Jensen's inequality $\beta_2-\beta_1^2\geq0$. Hence, every $\mathcal{G^{W,E}}$-distribution has non-negatively correlated components, see Lemma \ref{l:corr} $(\mathcal{W})$. It follows by Example \ref{ex:corr2} that $\mathcal{G^{W,E}}=\mathcal{G^{N,E}}\subseteq\mathcal{G^{N,X}}$ in dimension $d=2$. It remains to justify that $\mathcal{G^{N,E}}\supseteq\mathcal{G^{N,X}}$, meaning that each sequence $(1,b_1,b_2)\in\mathcal{LM}_{3}$ is infinitely extendible to a sequence $(1,b_1,b_2,\ldots)\in\mathcal{LM}_{\infty}$. Let $(1,b_1,b_2)\in\mathcal{LM}_{3}$. On the one hand, by 3-log-monotonicity of $(1,b_1,b_2)$, it holds that $b_1^2\leq b_2\leq b_1<1$. On the other hand, by Proposition \ref{extend}, $(1,b_1,b_2)\in\mathcal{LM}_{3}$ is infinitely extendible to a sequence $(1,b_1,b_2,\ldots)\in\mathcal{LM}_{\infty}$ if and only if $(1,b^r_1,b^r_2)$ is infinitely extendible to a sequence $(1,b^r_1,b^r_2,\ldots)\in\mathcal{M}_{\infty}$ for each fixed $r>0$. Finally, Hankel's criterion for infinite extendibility of $d$-monotone sequences, see \cite{dette97}, p.\ 20, implies that $(1,b^r_1,b^r_2)$ is infinitely extendible to $(1,b^r_1,b^r_2,\ldots)\in\mathcal{M}_{\infty}$ if and only if $b_1^{2r}\leq b_2^r\leq b_1^r\leq1$ for each fixed $r>0$. This gives the assertion.           
\end{enumerate}
\end{remark}

In dimension $d\geq3$, none of the inclusions $\mathcal{G^{W,E}}\subseteq\mathcal{G^{N,X}}$ and $\mathcal{G^{N,X}}\subseteq\mathcal{G^{W,E}}$ holds. On the one hand, $\mathcal{G^{W,E}}\nsubseteq\mathcal{G^{N,X}}$ by Example \ref{ex:corr3}. On the other hand, it can be shown that not every $(1,b_1,b_2,b_3)\in\mathcal{LM}_{4}$ is extendible to a sequence $(1,b_1,b_2,b_3,\ldots)\in\mathcal{M}_{\infty}$.

The class of exchangeable narrow-sense geometric distributions possesses a strong positive dependence property, namely the multivariate right tail increasing (MRTI) dependence. MRTI implies, in particular, the positive upper orthant increasing ratio dependence, the positive upper orthant dependence, and, hence, non-negative correlations. For a detailed discussion of the notions of multivariate positive dependence, we refer to \cite{block81,tong94,colangelo05}. For the convenience of the reader, we recall that a random vector $(\tau_1,\ldots,\tau_d)$ with values in $\IN^d$ is said to be MRTI if for all permutations $\pi$: $\{1,\ldots,d\}\rightarrow\{1,\ldots,d\}$
\begin{gather*}
\IP(\tau_{\pi(i)}>n_i \,|\,\tau_{\pi(1)}>n_1,\ldots,\tau_{\pi(i-1)}>n_{i-1})
\end{gather*} 
is increasing in $n_1,\ldots,n_{i-1}$ for all $n_i$, $i\in\{2,\ldots,d\}$. 

We show in Theorem \ref{mrti_exch} that
$$\frac{a_k}{a_{k+1}}=\frac{b_k^2}{b_{k-1}b_{k+1}}\leq1\quad (\text{resp.\ }\frac{\beta_k^2}{\beta_{k-1}\beta_{k+1}}\leq1),\quad k=1,\ldots,d-1,$$
is a necessary and sufficient condition on the parameters of a $d$-dimensional $\mathcal{G^{N,X}}$- (resp.\ $\mathcal{G^{W,X}}$-) survival function, given by (\ref{Geo_surv_new}) (resp.\ (\ref{Geo_surv_new_w})), to be MRTI. 
In view of Theorem \ref{exch} $(\mathcal{N})$, it holds that $\{a_k\}_{k=1}^d\in\mathcal{SM}_d$ for any set of parameters $\{a_k\}_{k=1}^d$ of a $d$-variate $\mathcal{G^{N,X}}$-distribution. Hence, $\{a_k\}_{k=1}^d$ is non-decreasing, and it follows that every $d$-variate $\mathcal{G^{N,X}}$-distribution is MRTI positively dependent. The next theorem summarizes these results.

\begin{theorem}[MRTI property for exchangeable distributions]\label{mrti_exch}\text{}
\vspace{-0.15cm}
\begin{enumerate}
	\item[($\mathcal{N}$)] Every $\mathcal{G^{N,X}}$-distributed $(\tau_1,\ldots,\tau_d)$ is MRTI positively dependent. 
	\item[($\mathcal{W}$)] A $\mathcal{G^{W,X}}$-distributed $(\tau_1,\ldots,\tau_d)$ is MRTI positively dependent if and only if the parameters $\{\beta_k\}_{k=0}^d$ of its corresponding survival function $\bar{F}^{\mathcal{W}}$, given by (\ref{Geo_surv_new_w}), satisfy $\frac{\beta_k^2}{\beta_{k-1}\beta_{k+1}}\leq1$ for all $k=1,\ldots,d-1$.
\end{enumerate}
\end{theorem}
\begin{proof} Refer to \ref{sec:appd}.\qed
\end{proof}

On the one hand, exchangeable wide-sense geometric distributions possess an MRTI positive dependence structure under the conditions given in Theorem \ref{mrti_exch} $(\mathcal{W})$. On the other hand, as illustrated in Example \ref{ex:nw2}, exchangeable wide-sense geometric distributions in general allow for negative correlations. Extendible (narrow- and wide-sense) geometric distributions are always MRTI. This is the content of Corollary \ref{mrti_ext}.

\begin{corollary}[MRTI property for extendible distributions]\label{mrti_ext}
Every extendible geometric distribution satisfying the local LM property is MRTI positively dependent. 
\end{corollary}
\begin{proof} In view of Corollary \ref{ext} and Theorem \ref{mrti_exch} $(\mathcal{W})$, it remains to show that any $\{\beta_k\}_{k\in\IN_0}\in\mathcal{M}_{\infty}$ satisfies  $\beta_k^2\leq\beta_{k-1}\beta_{k+1}$ for all $k\in\IN$. By Hausdorff's moment problem, $\{\beta_k\}_{k\in\IN_0}\in\mathcal{M}_{\infty}$ is equivalent to $\beta_k=\IE[\tau^k]$ for $k\in\IN_0$ and a $[0,1]$-valued random variable $\tau$. The Cauchy-Schwarz inequality gives $\IE[\tau^k]^2\leq\IE[\tau^{k-1}]\IE[\tau^{k+1}]$ for all $k\in\IN$, and the corollary follows.\qed
\end{proof}

\appendix
\section{Proofs from Section \ref{sec:mgd}}
\label{sec:appa}

\hspace{-0.525cm}\textsc{Proof of Theorem \ref{thm_a}}. To prove $(\mathcal{N})$, let $(\tau_1,\ldots,\tau_d)\sim\mathcal{G^N(\textbf{p})}$. Check that
\begin{gather*}
\IP(\tau_1\hspace{-0.1cm}>\hspace{-0.1cm}n_1,\ldots,\tau_d\hspace{-0.1cm}>\hspace{-0.1cm}n_d)=\IP \big(E_I\hspace{-0.1cm} > \hspace{-0.1cm}\max_{i \in I}\{n_i\},\,\emptyset\hspace{-0.1cm} \neq \hspace{-0.1cm}I \hspace{-0.1cm}\subseteq \hspace{-0.1cm}\{1,\ldots,d\} \big) =\hspace{-0.7cm}\prod_{\emptyset \neq I \subseteq \{1,\ldots,d\}}\hspace{-0.5cm}p_I^{\max_{i \in I}\{n_i\}}.
\end{gather*}
Next, we show that the claimed survival function (\ref{Geo_surv}) satisfies the local LM property. Fix $k\in\{1,\ldots,d\}$ and let $1\leq i_1<\ldots<i_k \leq d,\,m,n_{i_1},\ldots,n_{i_k} \in \IN_0$. Define $\tilde{n}_1,\ldots,\tilde{n}_d$ by 
\begin{gather*}
\tilde{n}_j:=\left\{ \begin{array}{ll}
    n_j,        &\mbox\; j \in \{i_1,\ldots,i_k\}\\
    0,        &\mbox\; j \notin \{i_1,\ldots,i_k\}\\
\end{array}\right..
\end{gather*}
Then,
\begin{gather*}
\prod_{\emptyset \neq I \subseteq \{1,\ldots,d\}}\hspace{-0.3cm} p_{I}^{\max_{j \in I}\{\tilde{n}_j+m\}}=\Big(\prod_{\emptyset \neq I \subseteq \{1,\ldots,d\}}\hspace{-0.3cm} p_{I}^{\max_{j \in I}\{\tilde{n}_i\}}\Big)\,\Big(\prod_{\emptyset \neq I \subseteq \{1,\ldots,d\}}\hspace{-0.3cm} p_{I}^{\max_{j \in I}\{m\}}\Big),
\end{gather*}
which implies the assertion.\\
To prove $(\mathcal{W})$, let $(\tau_1,\ldots,\tau_d)\sim\mathcal{G^W(\tilde{\textbf{p}})}$. The event $\{\tau_1>n_1,\ldots,\tau_d>n_d\}$ means that the first $n_{(1)}$ outcomes are of type $I=\emptyset$, the next $n_{(2)}-n_{(1)}$ outcomes are $I\in\{\emptyset,\{\pi_{\textbf{\textit{n}}}(1)\}\}$, the next $n_{(3)}-n_{(2)}$ outcomes are $I\in\{\emptyset,\{\pi_{\textbf{\textit{n}}}(1)\}\ , \\ \{\pi_{\textbf{\textit{n}}}(2)\}\ , \{\pi_{\textbf{\textit{n}}}(1),\pi_{\textbf{\textit{n}}}(2)\}\}$, etc. Thus, with $n_{(0)}:=0$ we have
\begin{gather*}
\IP(\tau_1>n_1,\ldots,\tau_d>n_d)=\prod_{k=1}^d\Big(\sum_{\stackrel{I\subseteq \{1,\ldots,d\}}{\pi_{\textbf{n}}(i)\notin I\;\forall\,i=k,\ldots,d}} \tilde{p}_{I}\Big)^{n_{(k)}-n_{(k-1)}}.
\end{gather*}
Finally, we show that the survival function (\ref{Geo_surv_w}) satisfies the local LM property. Again, fix $k\hspace{-0.1cm}\in\hspace{-0.1cm}\{1,\ldots,d\}$ and let $1\hspace{-0.1cm}\leq \hspace{-0.1cm}i_1\hspace{-0.1cm}<\hspace{-0.1cm}\ldots\hspace{-0.1cm}<\hspace{-0.1cm}i_k \hspace{-0.1cm}\leq \hspace{-0.1cm}d,\,m,n_{i_1},\ldots,n_{i_k}\hspace{-0.1cm} \in\hspace{-0.05cm} \IN_0$. With $\tilde{n}_1,\ldots,\tilde{n}_d$ defined as above
\begin{gather*}
\bar{F}^{\mathcal{W}}_{\tilde{n}_1+m,\ldots,\tilde{n}_d+m}=\Big(\hspace{-0.6cm}\sum_{\stackrel{I\subseteq \{1,\ldots,d\}}{\pi_{\textbf{n}}(i)\notin I\;\forall\,i=1,\ldots,d}} \hspace{-0.7cm}\tilde{p}_{I}\Big)^{m}\prod_{k=1}^d\Big(\hspace{-0.6cm}\sum_{\stackrel{I\subseteq \{1,\ldots,d\}}{\pi_{\textbf{n}}(i)\notin I\;\forall\,i=k,\ldots,d}}\hspace{-0.7cm} \tilde{p}_{I}\Big)^{\tilde{n}_{(k)}-\tilde{n}_{(k-1)}}\hspace{-0.1cm}=\bar{F}^{\mathcal{W}}_{m,\ldots,m}\bar{F}^{\mathcal{W}}_{\tilde{n}_1,\ldots,\tilde{n}_d}.
\end{gather*} 
Hence the theorem.\qed

\vspace{0.7cm}

\hspace{-0.525cm}\textsc{Proof of Theorem \ref{Uniq}.}
By Theorem \ref{thm_a}, the $d$-variate wide-sense geometric distribution satisfies the local discrete LM property (\ref{LM}). Let $(\tau_1,\ldots,\tau_d)$ be a random vector on $\IN^d$ satisfying the local discrete LM property (\ref{LM}). For $I\subseteq \{1,\ldots,d\}$, define
\begin{gather}
\tilde{p}_I=\IP\big(\{\tau_i>1\;\forall i\notin I\}\cap\{\tau_i=1\;\forall i\in I\}\big).
\label{pis}
\end{gather}
We show that $(\tau_1,\ldots,\tau_d)$ follows a $\mathcal{G^W}(\{\tilde{p}_I\}_{I\subseteq \{1,\ldots,d\}})$-distribution with survival function (\ref{Geo_surv_w}). First, we prove that $(i)$ for $k=1,\ldots,d$,
\begin{gather*}
\sum_{\stackrel{I\subseteq \{1,\ldots,d\}}{\pi(i)\notin I\;\forall\,i=k,\ldots,d}} \hspace{-0.7cm}\tilde{p}_I=\IP(\tau_{\pi(k)}>1,\ldots,\tau_{\pi(d)}>1),
\end{gather*} 
for any permutation $\pi$: $\{1,\ldots,d\}\rightarrow\{1,\ldots,d\}$, $(ii)$ $\sum_{I}\tilde{p}_{I}=1$, $(iii)$ $\tilde{p}_I \in [0,1]$, and $(iv)$ $\sum_{I:k \notin I}\tilde{p}_{I}<1$ for $k=1,\ldots,d$. To see that $(i)$ holds, check that for $k=1,\ldots,d$
\begin{align*}
&\ \sum_{\stackrel{I\subseteq \{1,\ldots,d\}}{\pi(i)\notin I\;\forall\,i=k,\ldots,d}}\hspace{-0.7cm} \tilde{p}_I \hspace{0.3cm} =\hspace{-0.4cm}\sum_{\stackrel{\hspace{0.1cm}I\subseteq \{1,\ldots,d\}}{\hspace{0.2cm}\pi(i)\notin I\;\forall\,i=k,\ldots,d}}\hspace{-0.7cm}\IP\big(\{\tau_i>1\;\forall i\notin I\}\cap\{\tau_i=1\;\forall i\in I\}\big)\\
&\ =\IP(\tau_1>1,\ldots,\tau_d>1)\\
&\ \quad+\hspace{-0.1cm}\sum_{n=1}^{k-1}\hspace{-0.1cm}\hspace{-0.55cm}\sum_{\hspace{0.7cm}1\leq i_1<\ldots<i_{n}\leq k-1}\hspace{-0.4cm}\hspace{-0.65cm}\IP\big(\big\{\tau_j\hspace{-0.1cm}>\hspace{-0.1cm}1\;\forall j\hspace{-0.1cm}\notin\hspace{-0.1cm}\{\pi(i_1),\ldots,\pi(i_{n})\}\big\}\hspace{-0.1cm}\cap\hspace{-0.1cm}\big\{\tau_j\hspace{-0.1cm}=\hspace{-0.1cm}1\;\forall j\hspace{-0.1cm}\in\hspace{-0.1cm}\{\pi(i_1),\ldots,\pi(i_{n})\}\big\}\big)\hspace{-0.08cm}\\
&\ =\IP(\tau_{\pi(k)}>1,\ldots,\tau_{\pi(d)}>1),
\end{align*}
where the last equality is due to
\begin{gather*}
\hspace{-1cm}\bigcup_{\hspace{0.9cm}1\leq i_1<\ldots<i_{n}\leq k-1,\atop n=1,\ldots,k-1} \hspace{-1.45cm}\Big\{\hspace{-0.1cm}\big\{\tau_j\hspace{-0.1cm}>\hspace{-0.1cm}1\;\forall j\hspace{-0.1cm}\notin\hspace{-0.1cm}\{\pi(i_1),\ldots,\pi(i_{n})\}\hspace{-0.05cm}\big\}\hspace{-0.1cm}\cap\hspace{-0.1cm}\big\{\tau_j\hspace{-0.1cm}=\hspace{-0.1cm}1\;\forall j\hspace{-0.1cm}\in\hspace{-0.1cm}\{\pi(i_1),\ldots,\pi(i_{n})\}\hspace{-0.05cm}\big\}\hspace{-0.1cm}\Big\}\hspace{-0.1cm}\cup\hspace{-0.1cm}\{\hspace{-0.05cm}\tau_1\hspace{-0.1cm}>\hspace{-0.1cm}1,\ldots,\tau_d\hspace{-0.1cm}>\hspace{-0.1cm}1\hspace{-0.05cm}\}\\
=\big\{\tau_i>1\;\forall i\notin\{\pi(1),\ldots,\pi(k-1)\}\big\}=\{\tau_{\pi(k)}>1,\ldots,\tau_{\pi(d)}>1\}.
\end{gather*} 
$(ii)$ follows immediately from (\ref{pis}) and the fact that $(\tau_1,\ldots,\tau_d)$ is $\IN^d$-valued:
\begin{gather*}
\sum_{I\subseteq\{1,\ldots,d\}}\hspace{-0.4cm}\tilde{p}_{I}=\hspace{-0.3cm}\sum_{I\subseteq\{1,\ldots,d\}}\hspace{-0.4cm}\IP\big(\{\tau_i\hspace{-0.1cm}>\hspace{-0.1cm}1\;\forall i\hspace{-0.1cm}\notin\hspace{-0.1cm} I\}\cap\{\tau_i\hspace{-0.1cm}=\hspace{-0.1cm}1\;\forall i\hspace{-0.1cm}\in \hspace{-0.1cm}I\}\big)=\IP(\tau_1\hspace{-0.1cm}\geq\hspace{-0.1cm}1,\ldots,\tau_d\hspace{-0.1cm}\geq\hspace{-0.1cm}1)\hspace{-0.1cm}=\hspace{-0.1cm}1.
\end{gather*} 
$(iii)$ is trivial, and $(iv)$ follows from the finiteness of the components $\tau_k$, $k=1,\ldots,d$. More precisely, using $(i)$, it follows for $k=1,\ldots,d$, that 
\begin{gather*}
\sum_{I:k \notin I}\tilde{p}_{I}=\IP(\tau_k>1)<1,
\end{gather*} 
since, otherwise, by the LM property, $\IP(\tau_k>n)=\IP(\tau_k>1)^n=1$ for all $n\in\IN_0$, contradicting $\tau_k<\infty$ almost surely. 
Finally, we show that $(\tau_1,\ldots,\tau_d)\sim\mathcal{G^W}(\{\tilde{p}_I\}_{I\subseteq \{1,\ldots,d\}})$ with survival function as in (\ref{Geo_surv_w}). By the local LM property (\ref{LM}),
$$\IP(\tau_{i_1}>n,\ldots,\tau_{i_k}>n)=\IP(\tau_{i_1}>1,\ldots,\tau_{i_k}>1)^n,$$ 
for $k=1,\ldots,d$, $1\leq i_1<\ldots<i_k \leq d$, and $n\in\IN_0$. From this, with $\pi_{\textbf{\textit{n}}}$: $\{1,\ldots,d\}\rightarrow\{1,\ldots,d\}$ being a permutation depending on $\textbf{\textit{n}}=(n_1,\ldots,n_d)$ such that $n_{\pi_{\textbf{n}}(1)}\leq n_{\pi_{\textbf{n}}(2)}\leq\ldots\leq n_{\pi_{\textbf{n}}(d)}$, we derive
\begin{align*}
&\ \IP\big(\tau_1>n_1,\ldots,\tau_d>n_d\big)=\IP\big(\tau_{\pi_{\textbf{n}}(1)}>n_{(1)},\ldots,\tau_{\pi_{\textbf{n}}(d)}>n_{(d)}\big)\\
&\ \;=\IP\big(\tau_{\pi_{\textbf{n}}(1)}\hspace{-0.1cm}>\hspace{-0.1cm}n_{(1)},\tau_{\pi_{\textbf{n}}(2)}\hspace{-0.1cm}>\hspace{-0.1cm}n_{(1)}\hspace{-0.1cm}+\hspace{-0.1cm}(n_{(2)}\hspace{-0.1cm}-\hspace{-0.05cm}n_{(1)}),\ldots,\tau_{\pi_{\textbf{n}}(d)}\hspace{-0.1cm}>\hspace{-0.1cm}n_{(1)}\hspace{-0.1cm}+\hspace{-0.1cm}(n_{(d)}\hspace{-0.1cm}-\hspace{-0.05cm}n_{(1)})\big)\\
&\ \;= \IP\big(\tau_{\pi_{\textbf{n}}(1)}\hspace{-0.1cm}>\hspace{-0.1cm}n_{(1)},\ldots,\tau_{\pi_{\textbf{n}}(d)}\hspace{-0.1cm}>\hspace{-0.1cm}n_{(1)}\big)\IP\big(\tau_{\pi_{\textbf{n}}(2)}\hspace{-0.1cm}>\hspace{-0.1cm}n_{(2)}\hspace{-0.1cm}-\hspace{-0.05cm}n_{(1)},\ldots,\tau_{\pi_{\textbf{n}}(d)}\hspace{-0.1cm}>\hspace{-0.1cm}n_{(d)}\hspace{-0.1cm}-\hspace{-0.05cm}n_{(1)}\big)\\
&\ \; = \IP\big(\tau_1>1,\ldots,\tau_d>1\big)^{n_{(1)}}\IP\big(\tau_{\pi_{\textbf{n}}(2)}>1,\ldots,\tau_{\pi_{\textbf{n}}(d)}>1\big)^{(n_{(2)}-n_{(1)})}\\
&\ \quad\;\times\IP\big(\tau_{\pi_{\textbf{n}}(3)}>n_{(3)}-n_{(2)},\ldots,\tau_{\pi_{\textbf{n}}(d)}>n_{(d)}-n_{(2)}\big)\\
&\ \; = \prod_{k=1}^d \IP\big(\tau_{\pi_{\textbf{n}}(k)}\hspace{-0.1cm}>\hspace{-0.1cm}1,\ldots,\tau_{\pi_{\textbf{n}}(d)}\hspace{-0.1cm}>\hspace{-0.1cm}1\big)^{n_{(k)}-n_{(k-1)}}= \prod_{k=1}^d \big(\hspace{-0.3cm}\sum_{\stackrel{I\subseteq \{1,\ldots,d\}}{\pi_{\textbf{n}}(i)\notin I\;\forall\,i=k,\ldots,d}}\hspace{-0.5cm} \tilde{p}_{I}\big)^{n_{(k)}-n_{(k-1)}},
\end{align*}
where the last equality is due to $(i)$. Definition \ref{def} together with Theorem \ref{thm_a} give the result.\qed

\section{Proofs from Section \ref{sec:exch}}
\label{sec:appb}

\hspace{-0.525cm}\textsc{Proof of Lemma \ref{lemma_exch_case}.}
First, we prove the claim for the survival function in (\ref{Geo_surv}). Necessity is shown via induction over the cardinality of subsets of $\{1,\ldots,d\}$. To begin with, it is shown that $p_{\{1\}}=\ldots=p_{\{d\}}$. Exchangeability implies for each $k=2,\ldots,d$ that
\begin{gather*}
\prod_{\stackrel{\emptyset \neq I \subseteq \{1,\ldots,d\}}{I \neq \{1\}}}p_I = \bar{F}_{0,1,\ldots,1}=\bar{F}_{1,\ldots,1,0,1,\ldots,1}=\prod_{\stackrel{\emptyset \neq I \subseteq \{1,\ldots,d\}}{I \neq \{k\}}}p_I.
\end{gather*}
Division by $\prod_{\emptyset \neq I \subseteq \{1,\ldots,d\}}p_I=\bar{F}_{1,\ldots,1}$ on both sides gives $p_{\{1\}}=p_{\{k\}}$. Concerning the induction step, assume the claim to be true for cardinalities less than or equal to $k<d$ and let $I_0 \subseteq \{1,\ldots,d\}$ with $|I_0|=k+1$. Define 
\begin{gather*}
n_j:=\left\{ \begin{array}{ll}
    1,        &\mbox\; j \in I_0\\
    0,        &\mbox\; j \notin I_0\\
\end{array}\right.,\quad \tilde{n}_j:=\left\{ \begin{array}{ll}
    1,        &\mbox\; j \in \{1,\ldots,k+1\}\\
    0,        &\mbox\; j \notin \{1,\ldots,k+1\}\\
\end{array}\right..
\end{gather*}
From exchangeability it follows that
\begin{gather*}
\prod_{\stackrel{\emptyset \neq I \subseteq \{1,\ldots,d\}}{I \nsubseteq I_0}}p_I = \bar{F}_{n_1,\ldots,n_d}=\bar{F}_{\tilde{n}_1,\ldots,\tilde{n}_d}=\prod_{\stackrel{\emptyset \neq I \subseteq \{1,\ldots,d\}}{I \nsubseteq \{1,\ldots,k+1\}}}p_I.
\end{gather*}
Division by $\prod_{\emptyset \neq I \subseteq \{1,\ldots,d\}}p_I=\bar{F}_{1,\ldots,1}$ on both sides implies
\begin{gather*}
p_{I_0}\,\underbrace{\prod_{\stackrel{\emptyset \neq I \subsetneq I_0}{|I|\leq k}}p_I}_{=:(\ast)}=p_{\{1,\ldots,k+1\}}\,\underbrace{\prod_{\stackrel{\emptyset \neq I \subsetneq \{1,\ldots,k+1\}}{|I|\leq k}}p_I}_{=:(\ast\ast)},
\end{gather*}
where $(\ast)=(\ast \ast)$ holds by induction hypothesis. Hence, $p_{\{1,\ldots,k+1\}}=p_{I_0}$, and the necessity is shown.\\
To prove sufficiency, assume that the parameters satisfy (\ref{exch_condi}) and let $\pi$ denote any permutation on $\{1,\ldots,d\}$. Then, using the notation $\pi(I):=\{\pi(i):\,i\in I\}$,
\begin{align*}
\bar{F}_{n_{\pi(1)},\ldots,n_{\pi(d)}}&=\prod_{\emptyset \neq I \subseteq \{1,\ldots,d\}} p_{I}^{\max_{i \in I}\{n_{\pi(i)}\}}=\prod_{\emptyset \neq I \subseteq \{1,\ldots,d\}} p_{I}^{\max_{i \in \pi(I)}\{n_i\}}\\
&=\prod_{\emptyset \neq I \subseteq \{1,\ldots,d\}} p_{\pi(I)}^{\max_{i \in \pi(I)}\{n_i\}}=\bar{F}_{n_1,\ldots,n_d},
\end{align*} 
where the third equality above follows from the assumption (\ref{exch_condi}).
\par 
Now, we establish the claim for the survival function in (\ref{Geo_surv_w}). Necessity is shown, again, via induction over the cardinality of subsets of $\{1,\ldots,d\}$. From $\bar{F}_{0,1,\ldots,1}=\bar{F}_{1,0,1,\ldots,1}=\ldots=\bar{F}_{1,\ldots,1,0},$ it follows that $\tilde{p}_{\{1\}}=\ldots=\tilde{p}_{\{d\}}$. Assume the claim is true for cardinalities less than or equal to $k<d$. From exchangeability of $\bar{F}_{n_1,\ldots,n_d}$, it follows for all permutations $\pi$ on $\{1,\ldots,d\}$ that
\begin{align*}
\sum_{\stackrel{I\subseteq \{1,\ldots,d\}}{i\notin I\;\forall\,i=k+2,\ldots,d}} \hspace{-0.85cm}\tilde{p}_{I}&\ \hspace{-0.2cm}=\IP(\tau_1>0,\ldots,\tau_{k+1}>0,\tau_{k+2}>1,\ldots,\tau_d>1)\\
&\ \hspace{-0.2cm}=\IP(\tau_{\pi(1)}\hspace{-0.1cm}>\hspace{-0.1cm}0,\ldots,\tau_{\pi(k+1)}\hspace{-0.1cm}>\hspace{-0.1cm}0,\tau_{\pi(k+2)}\hspace{-0.1cm}>\hspace{-0.1cm}1,\ldots,\tau_{\pi(d)}\hspace{-0.1cm}>\hspace{-0.1cm}1)\hspace{-0.1cm}=\hspace{-1.2cm}\sum_{\stackrel{I\subseteq \{1,\ldots,d\}}{\pi(i)\notin I\;\forall\,i=k+2,\ldots,d}}\hspace{-0.85cm}\tilde{p}_{I}.
\end{align*} 
From this, using the induction hypothesis, we infer
\begin{align*}
\tilde{p}_{\{1,\ldots,k+1\}}&\ =\hspace{-0.4cm}\sum_{\stackrel{I\subseteq \{1,\ldots,d\}}{i\notin I\;\forall\,i=k+2,\ldots,d}} \hspace{-0.4cm}\tilde{p}_{I}\hspace{0.4cm}-\hspace{-0.4cm}\sum_{\stackrel{I\subseteq \{1,\ldots,d\},\;|I|\leq k}{i\notin I\;\forall\,i=k+2,\ldots,d}}\hspace{-0.4cm}\tilde{p}_{I}\hspace{0.4cm}\\
&\ =\hspace{-0.4cm}\sum_{\stackrel{I\subseteq \{1,\ldots,d\}}{\pi(i)\notin I\;\forall\,i=k+2,\ldots,d}}\hspace{-0.4cm}\tilde{p}_{I}\hspace{0.4cm}-\hspace{-0.4cm}\sum_{\stackrel{I\subseteq \{1,\ldots,d\},\;|I|\leq k}{\pi(i)\notin I\;\forall\,i=k+2,\ldots,d}}\hspace{-0.4cm}\tilde{p}_{I}=\tilde{p}_{\{\pi(1),\ldots,\pi(k+1)\}}.
\end{align*} 
Hence the necessity.\\
To prove sufficiency, assume that the parameters satisfy (\ref{exch_condi}) and let $\tilde{\pi}$ be any permutation on $\{1,\ldots,d\}$. Denoting $\tilde{\pi}(\textbf{\textit{n}})=(n_{\tilde{\pi}(1)},\ldots,n_{\tilde{\pi}(d)})$, and with $\hat{\pi}$: $\{1,\ldots,d\}\rightarrow\{1,\ldots,d\}$ being a permutation depending on $\tilde{\pi}(\textbf{\textit{n}})$ such that $n_{\hat{\pi}(\tilde{\pi}(1))}\leq n_{\hat{\pi}(\tilde{\pi}(2))}\leq\ldots\leq n_{\hat{\pi}(\tilde{\pi}(d))}$, we obtain 

\begin{align*}
\bar{F}_{n_{\tilde{\pi}(1)},\ldots,n_{\tilde{\pi}(d)}}&\ =\prod_{k=1}^d\Big(\sum_{\stackrel{I\subseteq \{1,\ldots,d\}}{\hat{\pi}(i)\notin I\;\forall\,i=k,\ldots,d}} \tilde{p}_{I}\Big)^{n_{\hat{\pi}}(k)-n_{\hat{\pi}}(k-1)}\\
&\ =\prod_{k=1}^d\Big(\sum_{\stackrel{I\subseteq \{1,\ldots,d\}}{\pi_{\textbf{n}}(i)\notin I\;\forall\,i=k,\ldots,d}} \tilde{p}_{I}\Big)^{n_{(k)}-n_{(k-1)}}=\bar{F}_{n_1,\ldots,n_d},
\end{align*} 
where the second equality above follows from assumption (\ref{exch_condi}) together with $n_{\hat{\pi}}(k)=n_{(k)}$ for all $k=1,\ldots,d$. Hence the lemma.\qed

\vspace{0.7cm}

\hspace{-0.525cm}\textsc{Proof of Theorem \ref{exch}.}
We first show $(\mathcal{N})$. From Theorem \ref{thm_a} $(\mathcal{N})$ together with (\ref{Geo_surv_new}), the $d$-dimensional non-degenerate exchangeable narrow-sense geometric survival function is given by 
\begin{gather*}
\bar{F}^{\mathcal{N}}_{n_1,\ldots,n_d}=\prod_{k=1}^{d}\Big(\prod_{i=1}^{d-k+1}p_i^{\binom{d-k}{i-1}}\Big)^{n_{(d-k+1)}}=\prod_{k=1}^{d}a_k^{n_{(d-k+1)}},
\end{gather*}
with $p_k\in(0,1]$ for all $k=1,\ldots,d$, $\prod_{i=1}^d p_i<1$, and $\{a_k\}_{k=1}^d$ defined as in (\ref{ak}). The formula for $p_k$, $k=1,\ldots,d$, is justified by
\begin{align*}
\prod_{i=1}^{k}\hspace{-0.05cm}a_{d-i+1}^{(-1)^{(k-i)}\binom{k-1}{i-1}}&\ \hspace{-0.3cm}=\prod_{i=1}^{k}\Big(\prod_{j=1}^i p_{j}^{\binom{i-1}{j-1}}\,\Big)^{(-1)^{(k-i)}\binom{k-1}{i-1}}=\prod_{i=1}^{k}\prod_{j=1}^i p_{j}^{(-1)^{(k-i)}\binom{k-j}{i-j}\binom{k-1}{j-1}}\\
&\ \hspace{-0.3cm}=\hspace{-0.1cm}\prod_{j=1}^{k}p_{j}^{(-1)^{(k)}\binom{k-1}{j-1}\sum_{i=j}^{k}(-1)^{(i)}\binom{k-j}{i-j}}\hspace{-0.2cm}=\hspace{-0.1cm}\prod_{j=1}^{k}p_{j}^{(-1)^{(k+j)}\binom{k-1}{j-1}\ind_{\{k=j\}}}\hspace{-0.1cm}=\hspace{-0.1cm}p_k,
\end{align*}
where the fourth equality follows from
$$\sum_{i=j}^{k}(-1)^{(i)}\binom{k-j}{i-j}=\sum_{l=0}^{k-j}(-1)^{(l+j)}\binom{k-j}{l}=(-1)^{(j)}\ind_{\{k=j\}}.$$
Next, we show that $p_k\in(0,1]$ for all $k=1,\ldots,d$ and $\prod_{i=1}^d p_i<1$ if and only if $\{a_k\}_{k=1}^d\in\mathcal{SM}_d$. To this end, observe that for $k=0,1,\ldots,d-1$,
\begin{align*}
&\ \nabla^{k-1}(\ln a_{d-k+1}^{-1}) = \sum_{i=0}^{k-1}(-1)^{i}\,\binom{k-1}{i}(-\ln a_{d-k+i+1})\\
&\  \quad=\sum_{i=0}^{k-1}(-1)^{k-i-1}\,\binom{k-1}{i}(-\ln a_{d-i})=\sum_{i=1}^{k}(-1)^{k-i}\,\binom{k-1}{i-1}(-\ln a_{d-i+1})\\ &\ \quad=\;-\ln\Big(\prod_{i=1}^k a_{d-i+1}^{(-1)^{(k-i)}\binom{k-1}{i-1}}\Big)=\;-\ln p_k, 
\end{align*}
and $\nabla^{k-1}\hspace{-0.05cm}(\ln\hspace{-0.01cm} a_{d-k+1}^{-1})\hspace{-0.1cm}\geq \hspace{-0.1cm}0$ for all $k\hspace{-0.1cm}=\hspace{-0.1cm}1,\ldots,d$ is equivalent to $\nabla^{d-k-1}\hspace{-0.05cm}(-\hspace{-0.1cm}\ln \hspace{-0.01cm}a_{k+1})\hspace{-0.1cm}\geq \hspace{-0.1cm}0$ for all $k=0,1,\ldots,d-1$. Taking into account that $$a_1=\prod_{i=1}^d p_i^{\binom{d-1}{i-1}},$$ it follows that $p_k\in(0,1]$ for all $k=1,\ldots,d$ and $\prod_{i=1}^d p_i<1$ if and only if $\{a_k\}_{k=1}^d\in\mathcal{SM}_d$. Furthermore, Theorem \ref{thm_a} $(\mathcal{N})$ together with (\ref{ak}) and (\ref{Geo_surv_new}) imply that from any finite sequence $\{a_k\}_{k=1}^d\in\mathcal{SM}_d$ one can construct an exchangeable $d$-variate narrow-sense geometric distribution by choosing the parameters $p_1,\ldots,p_d$ as 
\begin{gather*}
p_k:=\prod_{i=1}^k a_{d-i+1}\,^{(-1)^{(k-i)}\binom{k-1}{i-1}}, \quad k=1,\ldots,d.
\end{gather*} 
Finally, given $\{b_k\}_{k=1}^d$ as in (\ref{b's}), and denoting $b_0:=1$, observe that
\begin{gather*}
\nabla^{d-k}\ln a_k^{-1}=\nabla^{d-k}\ln b_{k-1}-\nabla^{d-k}\ln b_k=\nabla^{d-k+1}\ln b_{k-1}
\end{gather*}
for all $k=1,\ldots,d$. From this, we conclude that $\{b_k\}_{k=0}^d\in\mathcal{LM}_{d+1}$ if and only if $\{a_k\}_{k=1}^d\in\mathcal{SM}_d$. With (\ref{kurzi}), the assertion follows. 
\par
Next, we prove $(\mathcal{W})$. From Theorem \ref{thm_a} $(\mathcal{W})$ together with (\ref{Geo_surv_new_w}) and (\ref{beta's}), the $d$-dimensional exchangeable wide-sense geometric survival function is given by 
\begin{gather*}
\bar{F}^{\mathcal{W}}_{n_1,\ldots,n_d}=\prod_{k=1}^{d}\Big(\sum_{i=1}^{d-k+1}\binom{d-k}{i-1}\tilde{p}_i\Big)^{n_{(d-k+1)}-n_{(d-k)}}=\;\prod_{k=1}^{d}\beta_k^{n_{(d-k+1)}-n_{(d-k)}},
\end{gather*} 
with $\tilde{p}_1,\ldots,\tilde{p}_d\in [0,1]$, $\sum_{i=1}^{d}\binom{d}{i-1}\tilde{p}_i\leq1$, $\sum_{i=1}^{d}\binom{d-1}{i-1}\tilde{p}_i<1$, and $\{\beta_k\}_{k=1}^d$ defined as in (\ref{beta's}).
The formula for $\tilde{p}_k$, $k=1,\ldots,d$, is justified by
\begin{align*}
\nabla^{k-1}\beta_{d-k+1}&\ \hspace{-0.2cm}=\hspace{-0.1cm}\sum_{i=0}^{k-1}(-1)^i\hspace{-0.05cm}\binom{k-1}{i}\beta_{d-k+i+1}\hspace{-0.1cm}=\hspace{-0.1cm}\sum_{i=0}^{k-1}(-1)^i\hspace{-0.05cm}\binom{k-1}{i}\hspace{-0.1cm}\sum_{j=1}^{k-i}\hspace{-0.1cm}\binom{k-i-1}{j-1}\tilde{p}_j\\
&\ \hspace{-1.75cm}=\sum_{j=1}^{k}\sum_{i=0}^{k-j}(-1)^i\binom{k-1}{i}\binom{k-i-1}{j-1}\tilde{p}_j=\sum_{j=1}^{k}\binom{k-1}{j-1}\tilde{p}_j\sum_{i=0}^{k-j}(-1)^i\binom{k-j}{i}\\
&\ \hspace{-1.75cm}=\sum_{j=1}^{k}\binom{k-1}{j-1}\tilde{p}_j\ind_{\{k=j\}}=\tilde{p}_k.
\end{align*}
Now, we infer $(1,\beta_1,\ldots,\beta_d)\hspace{-0.1cm}\in\hspace{-0.1cm}\mathcal{M}_{d+1}$ from $\tilde{p}_1,\ldots,\tilde{p}_d\hspace{-0.1cm}\in \hspace{-0.1cm}[0,1]$ with $\sum_{i=1}^{d}\hspace{-0.1cm}\binom{d}{i-1}\tilde{p}_i\hspace{-0.1cm}\leq\hspace{-0.1cm}1$ and $\sum_{i=1}^{d}\binom{d-1}{i-1}\tilde{p}_i<1$, and vice versa. For $k=1,\ldots,d$, it holds that $\nabla^{d-k}\beta_k=\tilde{p}_{d-k+1}$. Further,
\begin{align*}
\nabla^d\beta_0 &\ \hspace{-0.2cm}=\hspace{-0.15cm}\sum_{j=0}^{d}(-1)^j\hspace{-0.05cm}\binom{d}{j}\beta_j\hspace{-0.1cm}=\hspace{-0.1cm}\beta_0\hspace{-0.1cm}+\hspace{-0.1cm}\sum_{j=1}^{d}(-1)^j\hspace{-0.05cm}\binom{d}{j}\beta_j\hspace{-0.1cm}=\hspace{-0.1cm}1\hspace{-0.1cm}+\hspace{-0.1cm}\sum_{j=1}^{d}(-1)^j\hspace{-0.05cm}\binom{d}{j}\hspace{-0.1cm}\sum_{i=1}^{d-j+1}\hspace{-0.15cm}\binom{d-j}{i-1}\tilde{p}_i\\
&\ \hspace{-0.2cm}=\hspace{-0.1cm}1\hspace{-0.1cm}+\hspace{-0.1cm}\sum_{i=1}^{d}\tilde{p}_i\hspace{-0.15cm}\sum_{j=1}^{d-i+1}\hspace{-0.1cm}(-1)^j\hspace{-0.05cm}\binom{d}{j}\binom{d-j}{i-1}\hspace{-0.1cm}=\hspace{-0.1cm}1\hspace{-0.1cm}+\hspace{-0.1cm}\sum_{i=1}^{d}\hspace{-0.1cm}\binom{d}{i-1}\tilde{p}_i\hspace{-0.1cm}\sum_{j=1}^{d-i+1}\hspace{-0.1cm}(-1)^j\hspace{-0.05cm}\binom{d-i+1}{j}\\
&\ \hspace{-0.2cm}=\hspace{-0.1cm}1\hspace{-0.1cm}+\hspace{-0.1cm}\sum_{i=1}^{d}\hspace{-0.05cm}\binom{d}{i-1}\tilde{p}_i\Big(\sum_{j=0}^{d-i+1}(-1)^j\hspace{-0.05cm}\binom{d-i+1}{j}\hspace{-0.1cm}-\hspace{-0.1cm}1\Big)\hspace{-0.1cm}=\hspace{-0.1cm}1\hspace{-0.1cm}-\hspace{-0.1cm}\sum_{i=1}^{d}\hspace{-0.05cm}\binom{d}{i-1}\tilde{p}_i.
\end{align*}
Lastly, $\beta_1=\sum_{i=1}^d\binom{d-1}{i-1}\tilde{p}_i$, and, therefore, $(1,\beta_1,\ldots,\beta_d)\in\mathcal{M}_{d+1}$ gives precisely the conditions on $\tilde{p}_1,\ldots,\tilde{p}_d$ as stated in the assertion. Furthermore, Theorem \ref{thm_a} $(\mathcal{N})$ together with (\ref{Geo_surv_new_w}) and (\ref{beta's}) imply that from any finite sequence $\{\beta_k\}_{k=1}^d$ with $(1,\beta_1,\ldots,\beta_d)\in\mathcal{M}_{d+1}$ one can construct an exchangeable $d$-dimensional wide-sense geometric distribution by choosing the parameters $\tilde{p}_1,\ldots,\tilde{p}_d$ as 
\begin{gather*}
\tilde{p}_k=\nabla^{k-1}\beta_{d-k+1}, \quad k=1,\ldots,d.
\end{gather*} This completes the proof.\qed
\vspace{0.7cm} 

\hspace{-0.525cm}\textsc{Proof of Theorem \ref{prop}.}
Sufficiency follows directly from Theorem~\ref{exch}~$(\mathcal{W})$. To prove necessity, suppose that~\eqref{Geo_surv_new} defines a $d$-dimensional survival function. \red{As is easily seen, it is exchangeable and satisfies the local discrete LM property~\eqref{LM}. Hence, in view of Theorem \ref{Uniq}, it is the survival function of a $d$-variate exchangeable wide-sense geometric distribution. The claim follows now from Theorem \ref{exch}~$(\mathcal{W})$.} \qed

\section{Proofs from Section \ref{sec:ext}}
\label{sec:appc}

\hspace{-0.525cm}\textsc{Proof of Proposition \ref{extend}.}
\red{That $(i)$ implies $(iv)$ follows directly from Corollary~\ref{ext}, recalling that $\mathcal{G^{N,E}}\subseteq\, \mathcal{G^{W,E}}$, together with the fact that $\mathcal{LM}_{\infty}$ is closed under positive powers. The reverse implication is easy: since $\{b_k\}_{k\in\IN_0}$ is strictly positive, it follows for all $k\in\IN_0$ and all $j\in\IN$ that 
 \begin{gather*}
 \nabla^{j}\ln b_k \,=\, \lim_{n\rightarrow\infty}n\,\nabla^{j}(\sqrt[n]{b_k}-1) \,=\, \lim_{n\rightarrow\infty}n\,\nabla^{j}\sqrt[n]{b_k}\,\geq\, 0.									
 \end{gather*}
}The equivalence of $(iii)$ and $(iv)$ is obtained from the closure of $\mathcal{M}_{\infty}$ under limits and products. \red{To see why $(ii)$ implies $(iv)$, note that from the infinite divisibility of $\kappa$, we have for all $n\hspace{-0.03cm}\in\hspace{-0.03cm}\IN$ that
\begin{gather}
\sqrt[n]{b_k} = \int_{[0,\infty]} e^{-k x} \kappa_n(d x), \quad k\in\IN_0,
\label{mu_n}
\end{gather} 
where $\kappa_n$ is such that $\kappa$ is the $n$-fold convolution on $\kappa_n$ with itself. One verifies using a simple induction that for all $n,j\in\IN$ and for all $k\in\IN_0$,     
\begin{gather*}
\nabla^{j}\hspace{-0.075cm}\sqrt[n]{b_k} \; =\; \nabla^{j-1}\hspace{-0.075cm}\sqrt[n]{b_k}-\nabla^{j-1}\hspace{-0.075cm}\sqrt[n]{b_{k+1}} \,  = \,  \int_{[0,\infty]} \hspace{-0.1cm}e^{-k x}\big(1-e^{-x}\big)^{j} \kappa_n(d x)\, \geq\, 0,
\end{gather*}
implying that $\{\sqrt[n]{b_k}\}_{k\in\IN_0}$ is completely monotone for all $n\in\IN$. Moreover, $b_0=1$ since $\kappa$ is a probability measure, $b_1<1$ because $\kappa$ is not concentrated at $0$, and $b_k>0$ for all $k\in\IN_0$ as $\kappa$ is not concentrated at $\infty$. Hence, $\{\sqrt[n]{b_k}\}_{k\in\IN_0}\in\mathcal{M}_{\infty}$ for all $n\in\IN$. Finally, to show that $(iv)$ implies $(ii)$, it is enough to notice that, in view of Hausdorff's moment problem \cite{hausdorff23}, $\{\sqrt[n]{b_k}\}_{k\in\IN_0}$ can be represented as the moment sequence (\ref{mu_n}) for some unique probability measure $\kappa_n$ on $[0,\infty]$ for all $n\in\IN$. Therefore, for all $n\in\IN$, $\kappa_1$ is the $n$-fold convolution of $\kappa_n$ with itself, meaning that $\kappa_1$ is infinitely divisible. And  $b_1\in(0,1)$ yields $\kappa(\{0\})<1$ and $\kappa(\{\infty\})<1$}.	 \qed 
\vspace{0.7cm}

\hspace{-0.525cm}\textsc{Proof of Theorem \ref{thm_ciid}.}
We first show that the model in $(\mathcal{W})$ generates precisely the family of $\mathcal{G^{W,E}}$-distributions. By Corollary \ref{ext} $(\mathcal{W})$, every $\mathcal{G^{W,E}}$-distribution is characterized by an infinite sequence $(1,\beta_1,\beta_2,\ldots)\in\mathcal{M}_{\infty}$. From the given sequence $(1,\beta_1,\beta_2,\ldots)\in\mathcal{M}_{\infty}$, we construct an infinite exchangeable sequence $\{\tau_k\}_{k \in \IN}$ of random variables such that $(\tau_1,\ldots,\tau_d)$ has the survival function (\ref{Geo_surv_new_w}) for each $d\geq2$. To this end, denote $\beta_0:=1$ and assume that $\{\beta_k\}_{k \in \IN_0}\in\mathcal{M}_{\infty}$. By virtue of Hausdorff's moment problem, see \cite{hausdorff21,hausdorff23}, there exists a unique probability law $\eta$ on the unit interval $[0,1]$ such that
\begin{gather*}
\beta_k=\int_{[0,1]}x^k\,\eta(dx),\quad k\in \IN_0.
\end{gather*}
Now consider a probability space $(\Omega,\F,\IP)$ supporting the following independent sequences of random variables:
\begin{itemize}
\item An i.i.d.\ sequence $\{Y_k\}_{k \in \IN}$, where $Y_1 \sim \eta$.
\item An i.i.d.\ sequence $\{E_k\}_{k \in \IN}$, where $E_1 \sim Exp(1)$.
\end{itemize}
Define the i.i.d.\ sequence $\{X_k\}_{k \in \IN}$ by $X_k:=-\ln(Y_k)$, $k \in \IN$. Notice that it may hold that $\IP(Y_k=0)>0$, in which case we have $\IP(X_k=\infty)>0$ by conveniently defining $\ln 0: = -\infty$. Define the sequence $\{\tau_k\}_{k \in \IN}$ by
\begin{gather*}
\tau_k:=\min\{n \in \IN\,:\,X_1+\ldots+X_n \geq E_k\},\quad k\in \IN.
\end{gather*} 
Notice that $\IP(X_k=0)<1$ by the assumption $\beta_1<1$, implying that the $\tau_k$ are well-defined in $\IN$. We compute for $n_1,\ldots,n_d \in \IN_0$ that
\begin{align*}
&\IP(\tau_1>n_1,\ldots,\tau_d>n_{d})=\IP(X_1+\ldots+X_{n_1}<E_1,\ldots,X_1+\ldots+X_{n_d}<E_d)\\
&\; =\IE\Big[\IP(X_1\hspace{-0.1cm}+\hspace{-0.1cm}\ldots\hspace{-0.1cm}+\hspace{-0.1cm}X_{n_1}\hspace{-0.1cm}<\hspace{-0.1cm}E_1,\ldots,X_1\hspace{-0.1cm}+\hspace{-0.1cm}\ldots\hspace{-0.1cm}+\hspace{-0.1cm}X_{n_d}\hspace{-0.1cm}<\hspace{-0.1cm}E_d\,|\,X_1,\ldots,X_{n_{(d)}})\Big]\\
&\; =\IE\Big[\prod_{k=1}^{d}\IP(X_1+\ldots+X_{n_k}<E_k\,|\,X_1,\ldots,X_{n_{(d)}})\Big]= \IE\Big[\prod_{k=1}^{d}e^{-(X_1+\ldots+X_{n_k})}\Big]\\
&\; = \IE\Big[ e^{-d\,(X_1+\ldots+X_{n_{(1)}})}\,e^{-(d-1)\,(X_{n_{(1)}+1}+\ldots+X_{n_{(2)}})}\,\cdots\,e^{-(X_{n_{(d-1)}+1}+\ldots+X_{n_{(d)}})} \Big]\\
&\; = \hspace{-0.1cm}\prod_{k=1}^{d}\hspace{-0.1cm}\IE\Big[e^{-(d-k+1)\,(X_{n_{(k-1)}+1}+\ldots+X_{n_{(k)}})} \Big]= \prod_{k=1}^{d}\IE\Big[\big(e^{-X_1}\big)^{(d-k+1)} \Big]^{n_{(k)}-n_{(k-1)}}\\
&\; = \hspace{-0.1cm}\prod_{k=1}^{d}\hspace{-0.1cm}\IE\Big[\big(e^{-X_1}\big)^{k} \Big]^{n_{(d-k+1)}-n_{(d-k)}}\hspace{-0.3cm}= \hspace{-0.1cm}\prod_{k=1}^{d}\hspace{-0.1cm}\IE\big[Y_1^{k} \big]^{n_{(d-k+1)}-n_{(d-k)}}\hspace{-0.1cm}=\hspace{-0.1cm}\prod_{k=1}^{d}\hspace{-0.1cm}\beta_k^{n_{(d-k+1)}-n_{(d-k)}}.
\end{align*}
Finally, by Corollary \ref{ext} $(\mathcal{N})$, the distribution of $(\tau_1,\ldots,\tau_d)$ is geometric in the narrow sense if and only if $(1,\beta_1,\beta_2,\ldots)\in\mathcal{LM}_{\infty}$. In view of Proposition \ref{extend}, it is the case precisely when the distribution of $X_1$ is infinitely divisible.\qed

\section{Proofs from Section \ref{sec:dep}}
\label{sec:appd}

\hspace{-0.525cm}\textsc{Proof of Theorem \ref{mrti_exch}.}
Let $\{\bar{F}_{n_1,\ldots,n_d}\}_{n_1,\ldots,n_d \in \IN_0}$, given by
\begin{gather*}
\bar{F}_{n_1,\ldots,n_d}=\prod_{k=1}^{d}a_k^{n_{(d-k+1)}}=\prod_{k=1}^{d}b_k^{n_{(d-k+1)}-n_{(d-k)}},
\end{gather*} be a $d$-dimensional survival function, where $b_0:=1$ and $\{b_k\}_{k=1}^d$ is as in (\ref{b's}). We show that $\{\bar{F}_{n_1,\ldots,n_d}\}_{n_1,\ldots,n_d \in \IN_0}$ is MRTI if and only if $\frac{a_k}{a_{k+1}}=\frac{b_k^2}{b_{k-1}b_{k+1}}\leq1$  $\forall k=1,\ldots,d-1$. 
\begin{gather}
\IP(\tau_k\hspace{-0.1cm} >\hspace{-0.1cm} n_k \,|\,\tau_1\hspace{-0.1cm} > \hspace{-0.1cm}n_1,\ldots,\tau_{k-1} \hspace{-0.1cm}> \hspace{-0.1cm}n_{k-1})\,\geq\,\IP(\tau_k \hspace{-0.1cm}>\hspace{-0.1cm} n_k \,|\,\tau_1 \hspace{-0.1cm}> \hspace{-0.1cm}\tilde{n}_1,\ldots,\tau_{k-1}\hspace{-0.1cm} >\hspace{-0.1cm} \tilde{n}_{k-1})
\label{ratio}
\end{gather} 
holds for all $n_i,\;\tilde{n}_i\in\IN_0$ with $n_i\geq\tilde{n}_i$, $i=1,\ldots,k-1$, and all $k=2,\ldots,d$ if and only if $a_k/a_{k+1}\leq1$. Fix an arbitrary $k\in\{2,\ldots,d\}$. Define the sequence $\{T_i\}_{i=1}^{k-1}$ recursively by  
\begin{align*}
T_i:=T_{i-1} &\ + \emph{sort}(n_1,\ldots,n_{k-1})[i] - \emph{sort}(\tilde{n}_1,\ldots,\tilde{n}_{k-1})[i]\\ 
&\ + \emph{sort}(\tilde{n}_1,\ldots,\tilde{n}_{k-1},n_k)[i] - \emph{sort}(n_1,\ldots,n_{k})[i], 
\end{align*}
where $T_0:=0$, and for $i\in\IN$, $i\leq k$, $\emph{sort}(n_1,\ldots,n_k)[i]$ denotes the i-th reverse order statistic of the set $\{n_1,\ldots,n_d\}$, so that $\emph{sort}(n_1,\ldots,n_k)[1]\geq \emph{sort}(n_1,\ldots,n_k)[2]\geq\ldots\geq \emph{sort}(n_1,\ldots,n_k)[k]$. Observe that 
\begin{align*}
T_{k-1} = &\ \sum_{i=1}^{k-1} \emph{sort}(n_1,\ldots,n_{k-1})[i] - \sum_{i=1}^{k-1} \emph{sort}(\tilde{n}_1,\ldots,\tilde{n}_{k-1})[i]\\
&\ + \sum_{i=1}^{k-1} \emph{sort}(\tilde{n}_1,\ldots,\tilde{n}_{k-1},n_k)[i] - \sum_{i=1}^{k-1} \emph{sort}(n_1,\ldots,n_{k})[i] \\
= &\  \sum_{i=1}^{k-1} n_i + \sum_{i=1}^{k-1} \tilde{n}_i + n_k - \sum_{i=1}^{k} n_i - \sum_{i=1}^{k-1} \tilde{n}_i\\
&\ + \emph{sort}(n_1,\ldots,n_{k})[k] - \emph{sort}(\tilde{n}_1,\ldots,\tilde{n}_{k-1},n_k)[k]\\ = &\ \emph{sort}(n_1,\ldots,n_{k})[k] - \emph{sort}(\tilde{n}_1,\ldots,\tilde{n}_{k-1},n_k)[k].
\end{align*}
This allows us to write the ratio of conditional probabilities in (\ref{ratio}) as
\begin{align*}
&\ \hspace{-0.15cm}\frac{\IP(\tau_k \hspace{-0.1cm}>\hspace{-0.1cm} n_k \,|\, \tau_1\hspace{-0.1cm} > \hspace{-0.1cm}n_1,\ldots,\tau_{k-1} \hspace{-0.1cm}>\hspace{-0.1cm} n_{k-1})}{\IP(\tau_k \hspace{-0.1cm}> \hspace{-0.1cm}n_k \,|\,\tau_1\hspace{-0.1cm}>\hspace{-0.1cm}\tilde{n}_1,\ldots,\tau_{k-1}\hspace{-0.1cm}>\hspace{-0.1cm} \tilde{n}_{k-1})}
\hspace{-0.1cm} =\hspace{-0.1cm}\frac{\prod_{i=1}^{k}\hspace{-0.1cm} a_i^{\emph{sort}(n_1,\ldots,n_{k})[i]}\prod_{i=1}^{k-1}\hspace{-0.1cm}a_i^{\emph{sort}(\tilde{n}_1,\ldots,\tilde{n}_{k-1})[i]}}{\prod_{i=1}^{k-1} \hspace{-0.1cm}a_i^{\emph{sort}(n_1,\ldots,n_{k-1})[i]}\prod_{i=1}^{k}\hspace{-0.1cm}a_i^{\emph{sort}(\tilde{n}_1,\ldots,\tilde{n}_{k-1},n_k)[i]}}\\
&\ \hspace{-0.15cm}=\hspace{-0.1cm} a_k^{\emph{sort}(n_1,\ldots,n_{k})[k] - \emph{sort}(\tilde{n}_1,\ldots,\tilde{n}_{k-1},n_k)[k]}\hspace{-0.1cm} \prod_{i=1}^{k-1}\hspace{-0.1cm} a_i^{T_{i-1}-T_i} \hspace{-0.1cm} =\hspace{-0.1cm}  \Big(\frac{a_2}{a_{1}}\Big)^{T_1}\hspace{-0.1cm} \Big(\frac{a_3}{a_{2}}\Big)^{T_2}\hspace{-0.2cm} \ldots\Big(\frac{a_k}{a_{k-1}}\Big)^{T_{k-1}}\hspace{-0.2cm}.
\end{align*}
It can be easily seen that the terms $T_1$ and $T_{k-1}$ are non-negative. We show that actually all elements of $\{T_i\}_{i=1}^{k-1}$ are non-negative, which results in 
\begin{gather*}
\IP(\tau_k \hspace{-0.1cm} >\hspace{-0.1cm}  n_k \,|\,\tau_1 \hspace{-0.1cm} >\hspace{-0.1cm}  n_1,\ldots,\tau_{k-1} \hspace{-0.1cm} > \hspace{-0.1cm} n_{k-1})\,\geq\,\IP(\tau_k \hspace{-0.1cm} > \hspace{-0.1cm} n_k \,|\,\tau_1 \hspace{-0.1cm} > \hspace{-0.1cm} \tilde{n}_1,\ldots,\tau_{k-1} \hspace{-0.1cm} > \hspace{-0.1cm} \tilde{n}_{k-1})
\end{gather*} 
if and only if $a_{i+1}/a_{i}\geq1$, $i=1,\ldots,k-1$, as claimed. The equivalence holds, since for each $i\in\{2,\ldots,k\}$,  
\begin{gather*}
\frac{\IP(\tau_k > n_k \,|\, \tau_1 > n_1,\ldots,\tau_{k-1} > n_{k-1})}{\IP(\tau_k > n_k \,|\,\tau_1>\tilde{n}_1,\ldots,\tau_{k-1}> \tilde{n}_{k-1})}
 =\frac{a_i}{a_{i-1}}
\end{gather*}
by choosing $n_j=\tilde{n}_{j}=2$ for all $j=1,\ldots,i-2$, $n_j=\tilde{n}_{j}=0$ for all $j=i,\ldots,k-1$, $n_{i-1}=n_k=1$, and $\tilde{n}_{i-1}=0$. 
\par
To complete the proof, we show that $\{T_i\}_{i=1}^{k-1}$ is non-negative. For this, order the sets $\{n_1,\ldots,n_{k}\}$ and $\{\tilde{n}_1,\ldots,\tilde{n}_{k-1},n_k\}$. Then, there exist $j,\;l\in\{1,\ldots,k\}$, $j\leq l$, such that
\begin{align*}
\;\qquad\qquad n_k&\ = \emph{sort}(\tilde{n}_1,\ldots,\tilde{n}_{k-1},n_k)[j],\\
\;\qquad\qquad n_k&\ = \emph{sort}(n_1,\ldots,n_{k})[l].
\end{align*}
Notice that 
\begin{align*}
\emph{sort}(\tilde{n}_1,\ldots,\tilde{n}_{k-1},n_k)[m]&\ = \emph{sort}(\tilde{n}_1,\ldots,\tilde{n}_{k-1})[m] \quad\qquad \forall \;m < j,\\
\emph{sort}(\tilde{n}_1,\ldots,\tilde{n}_{k-1},n_k)[m]&\ = \emph{sort}(\tilde{n}_1,\ldots,\tilde{n}_{k-1})[m-1] \;\quad \forall \;m > j,\\
\emph{sort}(n_1,\ldots,n_{k})[m]&\ = \emph{sort}(n_1,\ldots,n_{k-1})[m] \quad\qquad \forall \;m < l,\\
\emph{sort}(n_1,\ldots,n_{k})[m]&\ = \emph{sort}(n_1,\ldots,n_{k-1})[m-1] \;\quad \forall \;m > l.
\end{align*}
Hence, $T_i=0$ for all $i<j$. For $i \geq l$, we compute 
\begin{align*}
T_i  =&\ \sum_{m=1}^{i} \emph{sort}(n_1,\ldots,n_{k-1})[m] - \sum_{m=1}^{i} \emph{sort}(\tilde{n}_1,\ldots,\tilde{n}_{k-1})[m]\\ 
&\ + \sum_{m=1}^{i} \emph{sort}(\tilde{n}_1,\ldots,\tilde{n}_{k-1},n_k)[m] - \sum_{m=1}^{i} \emph{sort}(n_1,\ldots,n_{k})[m]\\
 =&\ \sum_{m=1}^{i} \emph{sort}(n_1,\ldots,n_{k-1})[m] - \sum_{m=1}^{i} \emph{sort}(\tilde{n}_1,\ldots,\tilde{n}_{k-1})[m]\\ 
&\ + \sum_{m=1}^{j-1} \emph{sort}(\tilde{n}_1,\ldots,\tilde{n}_{k-1})[m] + n_k + \sum_{m=j+1}^{i} \emph{sort}(\tilde{n}_1,\ldots,\tilde{n}_{k-1})[m-1]\\
&\ - \sum_{m=1}^{l-1} \emph{sort}(n_1,\ldots,n_{k-1})[m] - n_k - \sum_{m=l+1}^{i} \emph{sort}(n_1,\ldots,n_{k-1})[m-1]\\
=&\ \emph{sort}(n_1,\ldots,n_{k-1})[i] - \emph{sort}(\tilde{n}_1,\ldots,\tilde{n}_{k-1})[i]\; \geq \;0.
\end{align*}
Finally, if $j < l$, then, arguing analogously, we obtain for $j \leq i < l$,
\begin{gather*}
T_i = n_k - \emph{sort}(\tilde{n}_1,\ldots,\tilde{n}_{k-1})[i]\; \geq \; 0. 
\end{gather*} 
This verifies $(\mathcal{W})$. $(\mathcal{N})$ follows by noticing that $\{a_k\}_{k=1}^d$ is non-decreasing for any set of parameters $\{a_k\}_{k=1}^d$ of a $d$-variate $\mathcal{G^{N,X}}$-distribution.\qed

\bibliographystyle{elsarticle-num}

\end{document}